%% file: main-arxiv.tex
\documentclass[twoside]{article}

\usepackage[accepted]{aistats2020}
\usepackage[numbers]{natbib}
\usepackage{graphicx}
\usepackage{amsmath}
\usepackage{amssymb}
\usepackage{amsthm}
\usepackage{subfigure}
\usepackage{algorithm}
\usepackage{algorithmic}
\usepackage{hyperref}
\usepackage{color}
\usepackage{lipsum}
\usepackage{enumitem}
\usepackage{pgfplots}
\usepackage{epstopdf}
\usepackage{makecell}
\usepackage{nicefrac}
\usepackage{microtype}
\usepackage{booktabs} 
\usepackage{wrapfig}
\usepackage{tabularx}

\input{shortcut.tex}

\begin{document}

\twocolumn[

\aistatstitle{Optimization of Graph Total Variation
	via Active-Set-based Combinatorial Reconditioning}

\runningtitle{Optimization of Graph Total Variation 
	via Active-Set-based Combinatorial Reconditioning}

\aistatsauthor{ Zhenzhang Ye \And Thomas M\"ollenhoff \And  Tao Wu \And Daniel Cremers }

\aistatsaddress{ TU Munich \\ \href{mailto:zhenzhang.ye@tum.de}{zhenzhang.ye@tum.de} \And TU Munich \\ \href{mailto:thomas.moellenhoff@tum.de}{thomas.moellenhoff@tum.de} \And TU Munich \\ \href{mailto:tao.wu@tum.de}{tao.wu@tum.de} \And TU Munich \\ \href{mailto:cremers@tum.de}{cremers@tum.de}} ]

\begin{abstract}
Structured convex optimization on weighted graphs finds numerous applications in machine learning and computer vision. In this work, we propose a novel adaptive preconditioning strategy for proximal algorithms on this problem class. Our preconditioner is driven by a sharp analysis of the local linear convergence rate depending on the ``active set" at the current iterate. We show that nested-forest decomposition of the inactive edges yields a guaranteed local linear convergence rate. Further, we propose a practical greedy heuristic which realizes such nested decompositions and show in several numerical experiments that our reconditioning strategy, when applied to proximal gradient or primal-dual
hybrid gradient algorithm, achieves competitive performances. Our results suggest that local convergence analysis can serve as a guideline for selecting variable metrics in proximal algorithms.
\end{abstract}
\input{intro.tex}

\input{local_convergence.tex}

\input{condition_number.tex}

\input{implementation.tex}
\input{numerics.tex}
\input{conclusion.tex}

\bibliographystyle{abbrvnat}
\bibliography{biblio}

\input{supplementary.tex}

\end{document}

%% file: shortcut.tex
\newtheorem{theorem}{Theorem}
\newtheorem{corollary}[theorem]{Corollary}
\newtheorem{lemma}[theorem]{Lemma}
\newtheorem{proposition}[theorem]{Proposition}
\newtheorem{definition}[theorem]{Definition}

\newcommand{\bthm}{\begin{theorem}}
\newcommand{\ethm}{\end{theorem}}
\newcommand{\bcor}{\begin{corollary}}
\newcommand{\ecor}{\end{corollary}}
\newcommand{\blem}{\begin{lemma}}
\newcommand{\elem}{\end{lemma}}
\newcommand{\bprop}{\begin{proposition}}
\newcommand{\eprop}{\end{proposition}}
\newcommand{\bdefn}{\begin{definition}}
\newcommand{\edefn}{\end{definition}}
\newcommand{\bpf}{\begin{proof}}
\newcommand{\epf}{\end{proof}}

\newcommand{\iali}[1]{\begin{align}#1\end{align}}
\newcommand{\ialid}[1]{\begin{aligned}#1\end{aligned}}
\newcommand{\ieqn}[1]{\begin{equation}#1\end{equation}}

\DeclareMathOperator{\argmin}{arg\,min}

\DeclareMathOperator{\rnk}{rank}

\DeclareMathOperator{\spn}{span}

\DeclareMathOperator{\ran}{ran}

\DeclareMathOperator{\diag}{diag}

\DeclareMathOperator{\sgn}{sgn}

\DeclareMathOperator{\TV}{TV}

\DeclareMathOperator{\dist}{dist}

\newcommand{\bR}{\mathbb{R}}

\newcommand{\bN}{\mathbb{N}}

\newcommand{\cE}{\mathcal{E}}
\newcommand{\cT}{\mathcal{T}}

\newcommand{\cS}{\mathcal{S}}
\newcommand{\cA}{\mathcal{A}}
\newcommand{\cI}{\mathcal{I}}

\newcommand{\cV}{\mathcal{V}}

\newcommand{\cG}{\mathcal{G}}
\newcommand{\cO}{\mathcal{O}}

\renewcommand{\hat}{\widehat}
\renewcommand{\tilde}{\widetilde}

\newcommand{\norm}[1]{\| #1 \|}
\newcommand{\ip}[2]{\left\langle#1,#2\right\rangle}

\newcommand{\etal}{{et~al.}}

%% file: intro.tex
\section{Introduction}
\vspace{-0.1cm}

\emph{Preconditioning}, as a way of transforming a difficult linear system
into one that is easier to solve, enjoys a rich and successful history.
Recently, \emph{proximal algorithms} \cite{CoPe11,PaBo13,ChPo16b} have received a surge of popularity in solving
structured non-smooth convex optimization problems appearing across
many fields in science and engineering. Unlike in the case
of linear systems, putting forward a satisfactory theory
and implementation of preconditioning in the general non-smooth setting remains a largely unsolved challenge \cite{PoCh11,GiBo14a,GiBo14b,LSS14,BrSu15,FoBo15,GiBo15,ChPo16b,BeFaOchs18,MoYeWu18}.

This is mainly due to two obstacles:

\noindent
\textbf{(i)} The non-linear dynamics of proximal algorithms, as well as the geometry of the non-smooth energy are more involved than in the quadratic case.
     A precise characterization of the convergence behavior, which could guide the proper choice of metric (preconditioner), is challenging.

\noindent
\textbf{(ii)} In cases where the proper choice of metric is clear, non-diagonal preconditioners typically make the proximal operators in the algorithm much more expensive to evaluate. While favorably reducing the number of outer iterations, each inner iteration could be even of similar complexity as the original problem~\cite{LSS14}.

In this vein, numerous efforts have been devoted to a better understanding of the dynamics of proximal algorithms (see, e.g.,~\cite{nishihara15,garrigos17}), and exploring scenarios where non-diagonally scaled proximal mappings are still efficient to evaluate \cite{FrGo16,BeFa12,BeFaOchs18}.

In this paper we take a novel perspective, circumventing issue \textbf{(i)} by resorting to the local convergence analysis. This does not yield provable guarantees on the global iteration complexity. Nevertheless, we show empirically that our preconditioners guided by the local analysis yield an improvement long before the local linear convergence regime is entered (see Fig.~\ref{fig:theory}).

To overcome difficulty \textbf{(ii)} we restrict ourselves to
structured convex problems on weighted graphs, where metrics based on tree
decompositions are amenable to efficient proximal evaluation
thanks to recent message-passing algorithms~\cite{KPR16}.
Specifically, given an undirected weighted graph $\cG=(\cV,\cE,\omega)$, whose edges are weighted by a function $\omega: \cE\to\bR_{>0}$,
we consider the structured convex optimization on $\cG$:
\iali{
\min_{u\in\bR^\cV} G(u) + \TV_\cG(u),
\label{eq:pp}
}
where $\TV_\cG$ is the graph total variation
\ieqn{
\TV_\cG(u) = \sum_{e=(i, j) \in \cE} \omega_e \, | u_i - u_j |.
}
The function $G:\bR^\cV\to\bR\cup\{+\infty\}$ is assumed to be proper, lower semi-continuous and convex.

We define the vertex-to-edge map $K:\bR^\cV\to\bR^\cE$ by
\ieqn{
K=\diag(\omega)\nabla, \notag
\label{eq:wgrad}
}
where $\nabla$ is the (transposed) incidence matrix of $\cG$, i.e.,
\ieqn{
(\nabla u)_e = u_{i}-u_{j}, \quad \forall e=(i,j)\in\cE, \notag
\label{eq:wtv}
}
with arbitrarily fixed orientation. With this notation we can succinctly write $\TV_\cG(u) = \norm{Ku}_1$.

Problems of form \eqref{eq:pp} are, for example, relevant in image processing and computer vision~\cite{gilboa2008nonlocal,lou2010image,ChPo11,newcombe2011dtam}, unsupervised and transductive learning~\cite{hein2011beyond,hein2013total,bresson2013multiclass,garcia2014multiclass}, collaborative filtering~\cite{benzi2016song} and clustering~\cite{garcia2014multiclass}.

For separable convex $G(u) = \sum_{i \in \cV} ~ g_i(u_i)$, problem \eqref{eq:pp}
can be efficiently solved (up to machine precision) in polynomial time by
parametric max-flow methods
\cite{chambolle2009total,hochbaum2001efficient}. To handle
non-separable but differentiable $G$, the authors in
\cite{xin2014efficient} propose a (primal) proximal gradient iteration,
reducing \eqref{eq:pp} to a sequence of separable problems which are solved
by parametric max-flow. 
For problems on regular grids, several authors proposed a splitting into \emph{chains},
leading to 1D total variation subproblems which can be solved efficiently \cite{Con13,BaSr14,KPR16}.
An active-set method for submodular minimization (which includes the
graph total variation as a special case) was proposed in~\cite{KB17},
which is different from the active-set strategy pursued here.
Landrieu~\etal~recently proposed a fast method for graph total
variation~\cite{landrieu17,raguet18} by assuming that the solution is piecewise constant
and refining that partition by solving a sequence of max-flow problems.
Closely related to the present approach are projected Newton methods \cite{schmidt2012}, which
have also been applied to the total variation \cite{barbero2011fast}.

In contrast, the main focus of this paper is to advance the understanding of
preconditioning in proximal algorithms.
To solve the problem class \eqref{eq:pp}, we consider two types of algorithms:

\paragraph{(1) (Dual) proximal gradient (PG).} Assume $G^*$ is $C^2$ such that $l_{G^*} I \preceq \nabla^2 G^*(\cdot) \preceq L_{G^*} I$ for some constants $l_{G^*},L_{G^*}>0$.
Based on the (Fenchel) dual formulation of \eqref{eq:pp}, written
\iali{
\min_{p\in\bR^\cE} G^*(-K^\top p) +\delta\{\norm{p}_\infty \leq 1\},
\label{eq:dp}
}
one can apply the proximal (or projected) gradient:
\iali{
  p^{k+1} &= \arg\min_{p\in\bR^\cE} -\ip{K \nabla G^*(-K^\top p^k)}{p} \notag\\
  &\qquad +\delta\{\norm{p}_\infty \leq 1\} + \frac{t}{2} \norm{p - p^k}_{T_k}^2.
  \label{eq:ppg}
}
Here
$T_k \in\bR^{|\cE|\times|\cE|}$ is a symmetric positive definite matrix which induces a scaled norm $\|\cdot\|_{T_k}$ defined by $\|u\|_{T_k}^2={\ip{u}{u}_{T_k}}=u ^\top T_k u$.

\paragraph{(2) Primal-dual hybrid gradient (PDHG).}
Another equivalent formulation of \eqref{eq:pp} is the following convex-concave saddle-point problem:
\iali{
\min_{u\in\bR^\cV} \max_{p\in\bR^\cE} ~ \ip{Ku}{p} + G(u) - \delta\{\norm{p}_\infty \leq 1\} ,
\label{eq:spp}
}
to which one can apply the primal-dual hybrid gradient (PDHG) algorithm:
\iali{
u^{k+1} &= \arg\min_{u \in \bR^\cV}~ G(u) + \langle{p^k},{Ku}\rangle + \frac{s}{2}\norm{u-u^k}^2, \label{eq:ppdhg1}\\
p^{k+1} &= \arg\min_{p\in\bR^\cE} ~ -\ip{K(2u^{k+1}-u^k)}{p} \notag\\
 &\qquad +\delta\{\norm{p}_\infty \leq 1\}+\frac{t}{2}\norm{p-p^k}_{T_k}^2. \label{eq:ppdhg2}
}

\subsection{Related work on preconditioning}

The (vanilla) PG and PDHG (typically with $T_k \equiv I$), as special instances of proximal algorithms, are widely applied in convex optimization -- we refer to \cite{CoPe11,PaBo13,ChPo16b} for up-to-date surveys which contain relevant historical accounts and interconnection of algorithms. Acceleration of these algorithms is of significant research as well as practical interests.
To this end, momentum-based acceleration techniques, which are traced back to the seminal works by Nesterov \cite{Nes83} and Polyak \cite{Pol64}, were recently developed for PG \cite{BeTe09,OCBP14} and PDHG \cite{ChPo16a} and achieved impressive performances \cite{ChPo16b}.

In contrast to momentum methods, preconditioning techniques for proximal algorithms are less developed and understood, as previously discussed in the introduction. To clarify further, in the context of proximal methods there are roughly two separate streams of ideas referred to as preconditioning.

In the first one, the aim is to make the individual update steps in the algorithm easier while retaining a convergent method~\cite{BrSu15,ChPo11}. While making each iteration faster, the effect on the overall complexity is unclear.

The second line of works, aims at improving the theoretical convergence rate and thereby reducing the number of outer iterations, see~\cite{GiBo14a,GiBo14b,GiBo15}. However, these works make very restrictive assumptions on the problem class and do not apply to our setting. A consensus among these works is to minimize the \emph{(finite) condition number} $\kappa(T^{-1/2} K)$, which is defined by
\begin{equation}
  \kappa(\cdot) = \frac{\sigma_{\text{max}}(\cdot)}{\sigma_{\text{min}>0}(\cdot)},
  \label{eq:cond}
\end{equation}
as a reasonable heuristic in practice. This was (approximately) pursued for general problems in \cite{PoCh11,FoBo15,DiBo17}. In particular, forest-structured preconditioners for $K=\diag(\omega)\nabla$ which are provably optimal in terms of $\kappa(T^{-1/2} K)$ were proposed recently in \cite{MoYeWu18}.

%% file: local_convergence.tex
\section{Local convergence analysis} 
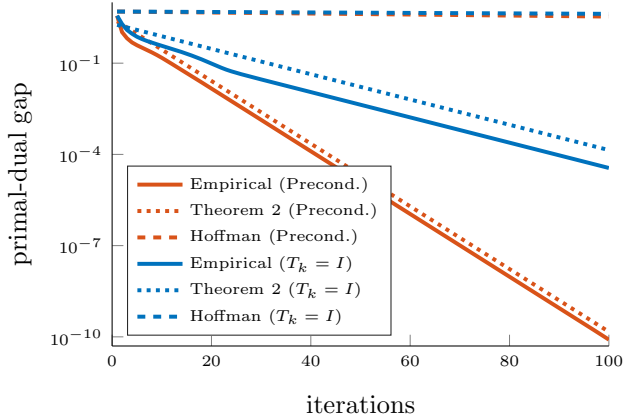
\begin{figure}
  \newcommand{\figurescale}{1}
  \input{hoffman.tex}
  \caption{Local vs global analysis of the linear convergence of the PG iteration \eqref{eq:ppg}. The local linear rate sharply matches the observed convergence behaviour, while the global rate based on Hoffman's bound is not informative. We guide the construction of our preconditioner based on the local convergence theory.}
  \label{fig:hoffman}
\end{figure}
While the condition number $\kappa(T^{-1/2} K)$ has proven to be reasonable heuristic in practice,
a more quantified connection between the convergence rate and the preconditioner $T$ would be desirable.

For problems of form \eqref{eq:pp}, global linear convergence of PG \eqref{eq:ppg} can be
established using Hoffman's bound~\cite{Hoffman52,klatte1995error,necoara2015linear,karimi2016linear}.  However, the
linear rate obtained from that bound is mainly of theoretical interest,
as it does not really inform us about the practical performance of the
method but rather gives a (weak) upper bound. Secondly, Hoffman's
bound is an inherently combinatorial expression that is very challenging to compute
even for small problem instances.

Instead, we aim to choose the preconditioner to improve the local
convergence behaviour of the method.  It turns out that for a wide
range of \emph{partly smooth} functions the local dynamics of the
PG, PDHG and accelerated variants thereof are well
understood, see~\cite{Liang14,Liang2015,liang2018local}. This
will serve as a basis for our theory.

In Fig.~\ref{fig:hoffman} we show the linear rate predicted by Hoffman's bound to the
local rate on a small $4 \times 3$ grid graph for which Hoffman's bound is still tractable to compute.  As discussed above, the global rate by Hoffman's bound is not
informative. 
The local analysis we present in 
Theorem~\ref{thm:conv} below (which proceeds similar to~\cite{Liang14}) is
sharp, matches the empirical performance and will be the guide of our preconditioners.

Next we establish the local linear convergence of \eqref{eq:ppg}.
Our strategy is to prove that, locally, iteration \eqref{eq:ppg} reduces to gradient descent on a modified unconstrained problem. Therefore, the local linear rate is inherited from the one of gradient descent.
\begin{lemma}
  Let $h$ be $C^2$ with $l_h I \preceq \nabla^2 h(\cdot) \preceq L_h I$ for some constants $l_h,\,L_h>0$. Then
 the gradient descent on $\min_x~h(Ax+b)$ with step size $1/t = 2 / (L_h \sigma_{\text{max}}(A)^2 + l_h \sigma_{\text{min}>0}(A)^2)$ satisfies
  \begin{equation}
    \norm{x^{k+1} - x^*} \leq \frac{\varphi - 1}{\varphi + 1} \norm{x^k - x^*},
  \end{equation}
  with $\varphi =  \kappa(A)^2 \cdot \kappa(h)$, $\kappa(h) := L_h / l_h$.
  \label{lm:cg}
\end{lemma}
\bpf
See the supplementary material.
\epf
The analysis in Theorem \ref{thm:conv} below hinges on finite identification of the \emph{active set} define as
\begin{equation}
  \cA(p) = \left \{ e \in \cE ~:~ |p_e| = 1 \right \}.
  \label{eq:aset}
\end{equation}
The associated projection matrix is defined as
\ieqn{
(P_\cA p)_e =
\begin{cases}
p_e & \text{if }e\in\cA, \\
0 & \text{if } e\notin\cA.
\end{cases}
}
Correspondingly, let $\cI(p):=\cE\backslash\cA(p)$ be the \emph{inactive set} and $P_\cI := I-P_\cA$.

\begin{theorem}
  Suppose that \eqref{eq:ppg} generates a sequence $\{p^k\}$ which converges to a minimizer $p^* \in \bR^\cE$ of \eqref{eq:dp}.
  Under the assumptions that
  \setlist[enumerate]{
    leftmargin=1cm
  }
  \begin{enumerate}[label=(A\arabic*)]
    \item \label{asm:a1} For each $e\in\cE$, $\left( K \nabla G^*(-K^\top p^*) \right)_e = 0 \Rightarrow |p_e^*| < 1$;
    \item \label{asm:a2} For each $k\in\bN$, $\underline{t} I \preceq T_k \preceq \bar{t} I$ with fixed $\underline{t},\,\bar{t} > 0$;
    \item \label{asm:a3} $T_k$ depends on $p^k$ only through $\cA(p^k)$;
  \end{enumerate}
  there exists $\bar{k} \in \bN$ such that for all $k \geq \bar{k}$:
  \setlist[enumerate]{
    leftmargin=0.75cm
  }
  \begin{enumerate}[label=(\roman*)]
    \item Finite identification, i.e.,
      \begin{equation}
        \cA(p^k) = \cA(p^*) \equiv \cA^*,~ T_k \equiv T.
      \end{equation}
    \item Local linear convergence, i.e.,
      \begin{equation}
        \norm{p^{k} - p^*}_T \leq \left( \frac{\varphi - 1}{\varphi + 1} \right)^{k-\bar{k}} \norm{p^{\bar{k}} - p^*}_T,
        \label{eq:local-conv}
      \end{equation}
      with
      \iali{ \label{eq:local-rate}
      \varphi &= \kappa(\Pi_{U(\cA^*)} T^{-1/2} K)^2 \cdot \kappa(G^*),
      }
      and $\Pi_{U(\cA^*)}$ the
      orthogonal projection onto the subspace $U(\cA^*) := \ker(P_{\cA^*} T^{-1/2})$.
    \end{enumerate}
    \label{thm:conv}
\end{theorem}
\begin{proof}

  (i) Finite identification of the active set follows by invoking \cite[Corollary~3.6]{burke88}. The strict complementary condition at $p^*$ required by that corollary is \ref{asm:a1}.
  Further, the corollary requires
  \begin{equation}
    \dist(0, \nabla J(p^k) + N(p^k)) \to 0,
    \label{eq:optzero}
  \end{equation}
  where $J = G^* \circ (-K^\top)$ and
  \begin{equation} \label{eq:ncdef}
    \begin{aligned}
      N(\bar p) = \Bigl \{ p \in \bR^\cE : p_e = 0 &~\text{ if } e \notin \cA(\bar p), \\
        \sgn(\bar p_e) \cdot p_e \geq 0 &~\text{ if } e \in \cA(\bar p) \Bigr \},
    \end{aligned}
  \end{equation}
  denotes the normal cone at $\bar p$. From the optimality conditions of \eqref{eq:ppg} it follows
    \iali{
      t T_k (p^{k} - p^{k+1}) - &\left( \nabla J(p^k) - \nabla J(p^{k+1}) \right) \notag\\
      &\in \nabla J(p^{k+1}) + N(p^{k+1}).
    }
  Then we have
    \iali{
        &\text{dist}(0, \nabla J(p^{k+1}) + N(p^{k+1})) \notag\\
        &\leq \norm{t T_{k} (p^{k} - p^{k+1}) - (\nabla J(p^{k}) - \nabla J(p^{k+1}))} \notag\\
        &\leq \left( t \norm{T_{k}} + L_{G^*} \lambda_{\text{max}}(K^\top K) \right) \norm{p^{k} - p^{k+1}}.
      }
    Convergence of $\{p^k\}$ to $p^*$ implies $\norm{p^{k} - p^{k+1}} \to 0$ and \eqref{eq:optzero} follows by
    \ref{asm:a2}. Since the active set is constant for $k \geq \bar{k}$ we have by \ref{asm:a3} that $T_k \equiv T$.

    (ii). Assume in the following that $k\geq \bar{k}$. Since $T_k = T$ due to (i), \eqref{eq:ppg} is equivalent the projected gradient descent applied to
  \begin{equation}
    \min_{q \in \bR^\cE} ~ \tilde J(q) \quad \text{ s.t.} ~\norm{T^{-1/2} q}_\infty \leq 1,
    \label{eq:psc}
  \end{equation}
  under the change of variable $p = T^{-1/2} q$, $\tilde J = J \circ T^{-1/2}$.
  The iteration in $q$ is given by
  \iali{
    & q^{k+1} = \underset{\substack{\norm{T^{-1/2} q}_\infty \leq 1}}{\argmin} ~ \langle \nabla \tilde J(q^k), q \rangle + \frac{t}{2} \norm{q - q^k}^2,
    \label{eq:ppgt}
  }
  whose optimality condition reads
  \iali{ \label{eq:qopt}
    t(q^{k}-q^{k+1}) \in T^{-1/2} N(T^{-1/2} q^{k+1}) + \nabla \tilde J(q^k).
  }
  From (i) we know that $P_{\cA^*} p^{k+1} = P_{\cA^*} p^k$, which yields
  \iali{
    &P_{\cA^*} T^{-1/2} q^{k+1} = P_{\cA^*} T^{-1/2} q^k, \notag \\
    &\Rightarrow q^{k+1} - q^k \in \ker(P_{\cA^*} T^{-1/2}) = U(\cA^*).
    \label{eq:z1}
  }
  In addition, 
  in view of \eqref{eq:ncdef}
  we have
  \iali{
  T^{-1/2}N(T^{-1/2}q^k) \subset U(\cA^*)^\perp. 
  }
  Thus, applying $\Pi_{U(\cA^*)}$ on both sides of \eqref{eq:qopt} yields an \emph{equivalent} characterization:
  \iali{
  0 &= \Pi_{U(\cA^*)}\nabla \tilde J(q^k) + t(q^{k+1} - q^k). \label{eq:qopt2}
  }
  Indeed, this is the gradient descent on $\tilde J$ restricted to $U(\cA^*)$, which we rewrite as
    \begin{align}
      q^{k+1} &= q^k + t^{-1} \Pi_{U(\cA^*)} T^{-1/2} K^\top \nabla G^*(-K^\top T^{-1/2} q^k) \notag\\
      &= q^k + t^{-1} \Pi_{U(\cA^*)} T^{-1/2} K^\top \nabla G^*(-K^\top T^{-1/2}  \notag\\
      &\quad (\Pi_{U(\cA^*)} q^k + \Pi_{U(\cA^*)^\perp} q^{\bar{k}})). \label{eq:ppgtl}
    \end{align}
  Hence \eqref{eq:ppgt} is equivalent to gradient descent
  on the function $G^* \circ (A \cdot \, + \, b)$ with $A = -K^\top T^{-1/2} \Pi_{U(\cA^*)}$, $b = -K^\top T^{-1/2} \Pi_{U(\cA^*)^\perp} p^{\bar{k}}$. Using Lemma~\ref{lm:cg} yields
  the linear convergence in $\{ q^k \}$. As $\norm{q^k} = \norm{T^{1/2} p^k} = \norm{p^k}_T$, we achieve the linear convergence, with respect to the $T$-norm, of the original sequence $\{p^k\}$. 
\end{proof}

\begin{corollary}
  Let $\varphi$ be given as in \eqref{eq:local-rate}.
  Locally (i.e., for $k \geq \bar{k}$), with fixed $T \equiv T_{\bar{k}}$ we have $\norm{p^k - p^*} \leq \varepsilon$ whenever
  \begin{equation}
    k \geq \bar{k} + \frac{\varphi + 1}{2} \log\left(\frac{\norm{p^{\bar{k}} - p^*} \sqrt{\kappa(T)}}{\varepsilon}\right).
    \label{eq:bound1}
  \end{equation}
  \label{cor:conv}
\end{corollary}
We remark that there are bounds in literature on $\bar{k}$, see
\cite[Prop.~3.6]{Liang2015} or the recent works~\cite{nutini17a,nutini17b}. Analyzing which choice of variable
metric $T_k$ lead to fast identification of $\cA^*$ is beyond the scope of this work.

%% file: hoffman.tex
%
%
\definecolor{mycolor1}{rgb}{0.00000,0.44700,0.74100}%
\definecolor{mycolor2}{rgb}{0.85000,0.32500,0.09800}%
\definecolor{mycolor3}{rgb}{0.92900,0.69400,0.12500}%
\definecolor{mycolor4}{rgb}{0.49400,0.18400,0.55600}%
\definecolor{mycolor5}{rgb}{0.46600,0.67400,0.18800}%
\definecolor{mycolor6}{rgb}{0.30100,0.74500,0.93300}%
\begin{tikzpicture}
\pgfplotsset{every tick label/.append style={font=\tiny}}

\begin{axis}[%
width=2.6in,
height=1.8in,
scale only axis,
xmin=0,
xmax=100,
ymode=log,
ymin=5e-11,
ymax=10,
yminorticks=true,
axis background/.style={fill=white},
legend style={at={(0.03,0.03)}, anchor=south west, legend cell align=left, align=left, draw=white!15!black},
scale=\figurescale,
axis x line*=bottom,
axis y line*=left,
ylabel={primal-dual gap},
xlabel={iterations},
ylabel near ticks,
]
\addplot [color=mycolor2, line width=1.5pt]
  table[row sep=crcr]{%
1	3.65020080667306\\
2	1.04970681656979\\
3	0.658845447937037\\
4	0.481016065810143\\
5	0.388986570456833\\
8	0.226223154635233\\
9	0.185673823596757\\
10	0.150269975997763\\
12	0.0952363287566187\\
14	0.059635024830537\\
18	0.0234705168090737\\
28	0.0021943194733604\\
97	1.66913401197872e-10\\
100	8.09363567936158e-11\\
};
\addlegendentry{\tiny Empirical (Precond.)}

\addplot [color=mycolor2, dotted, line width=1.5pt]
  table[row sep=crcr]{%
1	2.36602540378447\\
100	1.46871020407686e-10\\
};
\addlegendentry{\tiny Theorem~\ref{thm:conv} (Precond.)}

\addplot [color=mycolor2, dashed, line width=1.5pt]
  table[row sep=crcr]{%
1	4.98090581392968\\
100	3.41037806075657\\
};
\addlegendentry{\tiny Hoffman (Precond.)}

\addplot [color=mycolor1, line width=1.5pt]
  table[row sep=crcr]{%
1	3.65020080667318\\
2	1.9393524289374\\
3	1.23913796030391\\
4	0.931380348245684\\
5	0.740069460678068\\
6	0.621798629560345\\
7	0.537671347477873\\
9	0.419258868659763\\
14	0.232511324616552\\
16	0.178729710302665\\
18	0.133657458549144\\
20	0.0971355100397179\\
22	0.0705771589098194\\
23	0.061096727969639\\
24	0.0539143801893401\\
26	0.0433450520093945\\
30	0.0292033888205416\\
52	0.00354633103893025\\
100	3.57570017070538e-05\\
};
\addlegendentry{\tiny Empirical ($T_k = I$)}

\addplot [color=mycolor1, dotted, line width=1.5pt]
  table[row sep=crcr]{%
1	1.81734738579231\\
100	0.00013861987049576\\
};
\addlegendentry{\tiny Theorem~\ref{thm:conv} ($T_k = I$)}

\addplot [color=mycolor1, dashed, line width=1.5pt]
  table[row sep=crcr]{%
1	4.99118713778262\\
100	4.19135953321455\\
};
\addlegendentry{\tiny Hoffman ($T_k = I$)}

\end{axis}

\end{tikzpicture}%

%% file: condition_number.tex
\section{Combinatorial preconditioner}

Suggested by the local convergence analysis and Corollary~\ref{cor:conv} from the previous section, an ideal preconditioner $T$ ought to minimize the condition number $\kappa(\Pi_{U(\cA^*)} T^{-1/2} K)$ once the active set $\cA^*$ is identified. In practice, however, computationally amenable choices of $T$ are rather constrained due to a generic \emph{trade-off} between convergence speed of (outer) iterations and per-iteration cost, i.e., the $T$-scaled proximal evaluation in \eqref{eq:ppg} or \eqref{eq:ppdhg2}.
A dense matrix $T$, in general, will render inner iterations very expensive, as in the case of proximal Newton method \cite{LSS14}. For this reason, many authors consider diagonal preconditioners \cite{PoCh11,GiBo14a,GiBo14b} or (diagonal + low-rank) preconditioners \cite{BeFa12,BeFaOchs18} to keep the inner iterations fast and tractable.

Towards yet better balance of this trade-off, a recent paper \cite{MoYeWu18} makes use of fast TV solver on trees \cite{Con13,KPR16} and proposes a class of block diagonal preconditioners via graph partitioning (aiming at optimizing $\kappa(T^{-1/2} K)$ heuristically, however). There the optimal condition number $\kappa(T^{-1/2} K)$ is achieved by matroid partitioning. As a remark, combinatorial preconditioners for solving linear systems involving graph Laplacians date back to the early work by Vaidya in 1990's \cite{vaidya90}; refer to \cite{spielman2010algorithms} for a more detailed survey. 

In this section, we construct combinatorial preconditioners which are more faithful, compared to the ones from \cite{MoYeWu18}, to the (local) convergence analysis. In a nutshell, given the current active/inactive sets of edges, we partition the graph into \emph{inactively nested forests} in the sense of \eqref{eq:ninf}, so that the resulting preconditioner yields a guaranteed (local) convergence rate, which is made precise in Theorem \ref{thm:nf}.

To construct our preconditioner, let the edge set $\cE$ be partitioned into $L$ mutually disjoint subsets, i.e., $\cE=\bigsqcup_{l=1}^L\cE_l$, such that each subgraph $\cG_l=(\cV,\cE_l,\omega|_{\cE_l})$ is a {\it forest}.
Correspondingly, we define $P_l$ as the canonical projection from $\bR^\cE$ to $\bR^{\cE_l}$, i.e., $P_lp=p|_{\cE_l}$ for any $p\in\bR^\cE$.
Thus, the matrix $K$ can be decomposed into submatrices
$\{K_l\}_{l=1}^L$ 
where each $K_l=P_l K\in\bR^{|\cE_l|\times|\cV|}$. Analogously, let $\nabla_l = P_l \nabla$. 
Note that each $\nabla_l^\top$ (or $K_l^\top$) has full column rank, and hence 
\iali{
T_l:=K_l K_l^\top,\quad \forall l\in\{1,...,L\},
} 
is symmetric positive definite.

We then define our preconditioner as
\ieqn{ \label{eq:tmtx}
\ialid{
T &:= \sum_{l=1}^L P_l^\top T_l P_l.  
}}

In view of Theorem~\ref{thm:conv}, we analyze in the following the condition number of the following matrix:
\iali{
&\Pi_\cI := K^\top T^{-1/2} \Pi_{U(\cA)} T^{-1/2} K \notag\\
 &= K^\top T^{-1/2} (I - T^{-1/2}P_\cA(T^{-1/2}P_\cA)^\dagger) T^{-1/2}K.
\label{eq:pi-i}
}
As a preparatory result, the following lemma decomposes $\Pi_\cI$ into orthogonal projections onto subspaces.

\begin{lemma} \label{lem1}
Given $\cE=\cA\sqcup\cI$, let $\cG$ be partitioned into $L$ nonempty forests $\{\cG_l\}_{l=1}^{L}$. Then the matrix defined in \eqref{eq:pi-i} can be characterized as
\iali{
\Pi_\cI = \sum_{l=1}^L \Pi_{\cI,l}, \label{eq:pi-lem}
}
where each $\Pi_{\cI,l}$ is the orthogonal projection onto the linear subspace $\cS_{\cI,l}$ defined by
\iali{ \label{eq:ranil}
\cS_{\cI,l} := \spn\{\nabla^\top_e: e\in\cI\cap\cE_l\}.
}

\end{lemma}
\bpf
(i) We show the identity \eqref{eq:pi-lem} with
$I_l := P_l I P_l^\top$, $P_{\cA,l} := P_l P_\cA P_l^\top$, $P_{\cI,l} := P_l P_\cI P_l^\top$, and
\iali{
\Pi_{\cI,l} :=\, & K_l^\top T_l^{-1/2} (I_l-(T_l^{-1/2}P_{\cA,l})(T_l^{-1/2}P_{\cA,l})^\dagger) \notag\\
&T_l^{-1/2} K_l.
}
Note that 
\iali{
K^\top T^{-1/2} &=
\sum_{l=1}^L K^\top P_l^\top T_l^{-1/2} P_l \notag\\
&= \sum_{l=1}^L K_l^\top T_l^{-1/2} P_l, \label{eq:pi-il-1}\\
T^{-1/2} P_\cA &=
\left(\sum_{l=1}^L P_l^\top T_l^{-1/2} P_l \right)\left(\sum_{l'=1}^L P_{l'}^\top P_{\cA,l'} P_{l'} \right) \notag\\
&= \sum_{l=1}^L P_l^\top T_l^{-1/2} P_{\cA,l} P_l, \label{eq:pi-il-2}\\
(T^{-1/2} P_\cA)^\dagger &= \sum_{l=1}^L P_l^\top (T_l^{-1/2} P_{\cA,l})^\dagger P_l. \label{eq:pi-il-3}
}
By plugging \eqref{eq:pi-il-1}--\eqref{eq:pi-il-3} into \eqref{eq:pi-i}, we accomplish (i). 

(ii) We show each $\Pi_{\cI,l}$ is the orthogonal projection onto $\cS_{\cI,l}$.
First, it is easy to see $\Pi_{\cI,l}$ is symmetric and $\Pi_{\cI,l}^2=\Pi_{\cI,l}$, and hence an orthogonal projection. 
Secondly, note that $\rnk \Pi_{\cI,l}=|\cI\cap\cE_l|=\rnk K_l^\top P_{\cI,l}$. Furthermore, we have the following equation:
\iali{
\Pi_{\cI,l}K^\top_lP_{\cI,l} =\,& K_l^\top P_{\cI,l} - K_l^\top T_l^{-1/2} \notag\\
& (T_l^{-1/2}P_{\cA,l})(T_l^{-1/2}P_{\cA,l})^\dagger T_l^{1/2}P_{\cI,l} \notag\\
=\,& K_l^\top P_{\cI,l},
}
which completes step (ii).
\epf

\begin{theorem} \label{thm:nf}
Given $\cE=\cA\sqcup\cI$, let $\cG$ be partitioned into $L$ \emph{nonempty}, \emph{inactively nested} forests $\{\cG_l\}_{l=1}^{L}$ in the sense that
\iali{
\cS_{\cI,1} = ... = \cS_{\cI,{\hat l}} \supsetneq \cS_{\cI,{\hat l +1}} \supseteq ... \supseteq \cS_{\cI,L} \supsetneq \{0\}, \label{eq:ninf}
}
with the subspaces defined in \eqref{eq:ranil}.
Then we have $\lambda_{\min>0}(\Pi_\cI)=\hat l$ and the (local) convergence rate in Theorem~\ref{thm:conv} is $\varphi = (L / \widehat l) \cdot \kappa(G^*)$.
\label{thm:cn}
\label{thm:convrate}
\end{theorem}
\bpf
By Lemma \ref{lem1}, we have $\lambda_{\max}(\Pi_\cI)\leq \sum_{l=1}^L\lambda_{\max}(\Pi_{\cI,l}) \leq L$. In fact, the equality holds since $\Pi_\cI v=Lv$ for some nonzero $v\in\cS_{\cI,L}$. On the other hand, for any $v\in\ran\Pi_\cI$, we have $\ip{v}{\Pi_\cI v}\geq \sum_{l=1}^{\hat l} \ip{v}{\Pi_{\cI,l}v} = \hat{l}\|v\|^2$.  The equality holds for some nonzero $v\in \ran\Pi_\cI\cap (\cS_{\cI,\hat l+1})^\perp$.
This yields $\lambda_{\min>0}(\Pi_\cI)=\hat l$.
\epf

%% file: implementation.tex
\section{Implementation}
\label{sec:impl}
In this section, we specify 
how to construct our preconditioner and apply active-set-based reconditioning to both PG and PDHG algorithms. 
We only trigger reconditioning every $n$ iterations. 
When reconditioning is performed at iteration $k$,
a greedy heuristic is used for constructing the preconditioner $T_k$; see Section \ref{sec:recond}. 
Then, using the separability of $\norm{\cdot}_\infty$, 
we perform the updates across the subgraphs $\{\cG_l\}_{l=1}^L$:
\iali{
	p^{k+1}|_{\cE_l} &= \text{arg} \min_{p \in \bR^{\cE_l}} - \ip{K \bar{u}^k|_{\cE_l}}{p} \notag \\
										 &\qquad + \delta \{ \norm{p}_\infty \leq 1 \} + \frac{t}{2}\norm{p-p^k}^2_{T_{k,l}},
                                                                                 \label{eq:dual-upd}
}
where $\bar{u}^k$ is defined as:
\begin{equation}
	\bar{u}^k = \begin{cases}
		\nabla G^*(-K^\top p^k), & \text{for PG},\\
		2u^{k+1}-u^k, & \text{for PDHG}.
	\end{cases}
\end{equation}
The proximal evaluation required by \eqref{eq:dual-upd} is detailed in Section \ref{sec:dual-up}, which invokes the message-passing algorithm on trees. The overall complexity of the reconditioned algorithm is discussed in Section \ref{sec:complex}.

\subsection{Constructing preconditioner} \label{sec:recond}
Following Theorem~\ref{thm:cn} we aim to find a preconditioner $T_k$ which minimizes the condition number $\varphi 
= (L / \widehat l) \cdot \kappa(G^*)$, and hence the local linear convergence rate. 
Theoretically, optimal $T_k$ can be found in polynomial time by Matroid partitioning as in \cite{MoYeWu18}. The computation time is prohibitively large for the graphs in practical problems, however. 
Here we present a greedy heuristic to find inactively nested forests.

Given an input graph $\cG$ we partition the graph based on the active set at the current dual variable $p^k$. We assign to each edge $e \in \cE$ an additional weight $\rho_e = 1- \left| 1-\left| p^k_e \right| \right|$.
Then, a minimum spanning forest according to that weight is generated using Kruskal's algorithm~\cite{kruskal1956}. This spanning forest is then subtracted from current graph and will be added to the set $\{
\cG_l\}_{l=1}^L$. We perform this generation and subtraction iteratively until no edges remain in the original graph. 

The partitioning weight is introduced for two reasons: Firstly, we found it unstable to determine the active set $\cA(p^k)$ numerically according to a threshold;
Secondly, computing the preconditioner $T_k$ is quite expensive for large graphs. This strategy could extend the suitable duration of current preconditioner since a potential active edge often has a larger partitioning weight.

\subsection{Backward solver} \label{sec:dual-up}
The introduction of the proposed preconditioner $T_k$ makes the backward update \eqref{eq:dual-upd} more expensive. Here we describe how to solve it
efficiently, following the approach in~\cite{MoYeWu18}. 
Combining the linear and the quadratic term, \eqref{eq:dual-upd} can be re-written as:
\begin{equation}
	p^{k+1}|_{\cE_l}  = \text{arg} \min_{\norm{p}_\infty \leq 1} \frac{1}{2} \norm{K^\top_l p + f_l}^2,
        \label{eq:dual-upd2}
\end{equation}
where $f_l = -K_l^\top p^k|_{\cE_l} - \bar{u}^k/t$. The (Fenchel) dual problem of \eqref{eq:dual-upd2} is given by
\begin{equation}
	v_l = \text{arg}\min_{u \in \bR^\cV} \frac{1}{2}\norm{u-f_l}^2 + \norm{K_l u}_1,
        \label{eq:tv-on-tree}
\end{equation}
which is simply a weighted total variation problem on the individual trees in the forest $\cG_
l$. We solve the problem \eqref{eq:tv-on-tree} using the message-passing algorithm introduced in \cite{KPR16}.
To retrieve $p^{k+1}|_{\cE_l}$ from $v_l$ one can use the optimality condition:
\begin{equation}
	K_l^\top p^{k+1} |_{\cE_l} = v_l - f_l.
\end{equation}

\subsection{Discussion on complexity} \label{sec:complex}
For non-preconditioned proximal gradient, the complexity of each iteration
is $\cO( |\cE| )$. For the preconditioned variant it is $\cO( \sum_{t=1}^\cT |\cE_t| \log( |\cE_t|))$ where $\cT$ is the total number of trees
using the aforementioned message-passing algorithm~\cite{KPR16}. The preconditioned
update can still be parallelized to some extent, as the message-passing can run for each tree in parallel.

Construction of the preconditioner $T_k$ based on the greedy inactively nested forest strategy with Prim's or Kruskal's algorithm is $\cO( |\cE|^2 \log(|\cE|) / |\cV|)$ \cite{cheriton76}. 

After entering the local linear convergence phase,
the overall iteration complexity is $\mathcal{O}(\varphi \log(1/\varepsilon))$ to
find an $\varepsilon$-accurate solution (see Corollary~\ref{cor:conv}). While each iteration of the preconditioned algorithm
is slighty more costly (by roughly a factor of $\log(|\cE|)$), the condition number $\varphi$ 
is drastically reduced. For regular grids we have that $\varphi \in \cO(|\cV|)$ (cf. \cite[Theorem~4]{MoYeWu18}) in the non-preconditioned 
case. The proposed preconditioner improves this to a \emph{constant} $\varphi \in \cO(1)$, independent of problem size
at the expense of a slightly more expensive dual update step (up to a logarithmic factor).

%% file: numerics.tex
\newenvironment{customlegend}[1][]{%
  \begingroup
  \csname pgfplots@init@cleared@structures\endcsname
  \pgfplotsset{#1}%
}{%
  \csname pgfplots@createlegend\endcsname
  \endgroup
}%
\def\addlegendimage{\csname pgfplots@addlegendimage\endcsname}

\newcommand{\figurescale}{0.95}
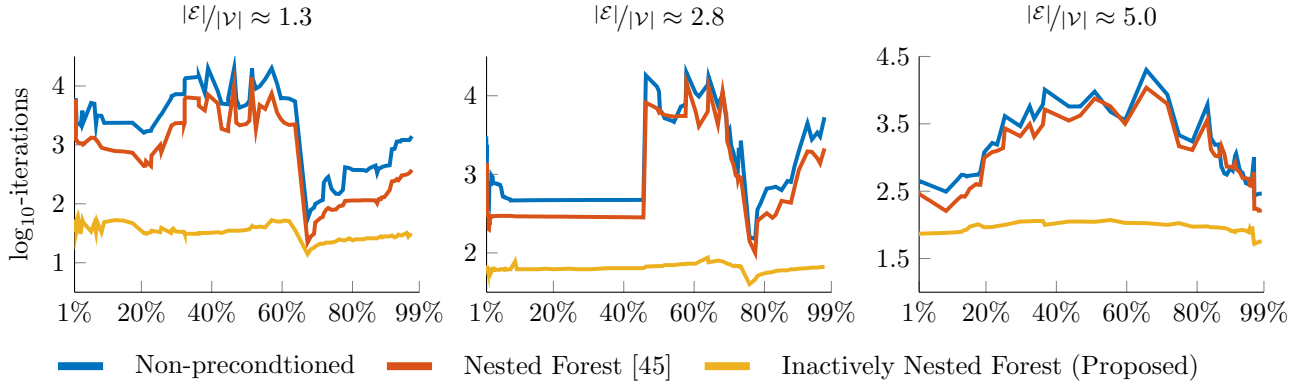
\begin{figure*}[t!]
  \centering
  \begin{tabular}{ccc}
    \input{active_sparse_new.tex} &\hspace{-0.5cm}
    \input{active_medium_new.tex} &\hspace{-0.5cm}
    \input{active_dense_new.tex} \\
  \end{tabular}
  \begin{tikzpicture}
    \begin{customlegend}[legend columns=3,legend style={align=center,draw=none,column sep=2ex},
      legend entries={Non-precondtioned,
        Nested Forest~\cite{MoYeWu18},
        Inactively Nested Forest (Proposed),
      }]
      \definecolor{mycolor1}{rgb}{0.00000,0.44700,0.74100}%
      \definecolor{mycolor2}{rgb}{0.85000,0.32500,0.09800}%
      \definecolor{mycolor3}{rgb}{0.92900,0.69400,0.12500}%
      \addlegendimage{mark=none,solid,line legend,line width=3pt,color=mycolor1}
      \addlegendimage{mark=none,solid,line width=3pt,color=mycolor2}   
      \addlegendimage{mark=none,solid,line width=3pt,color=mycolor3}   
    \end{customlegend}
  \end{tikzpicture}
	\caption{We show $\log_{10}$-iterations required by PG to reach a primal-dual gap smaller than $10^{-10}$ over percentage of active edges at the optimal solution for random graphs with varying edge-to-vertex ratio. The reconditioning strategy requires several orders of magnitude less iterations than no preconditioner and the preconditioner~\cite{MoYeWu18}. \vspace{0.2cm}}
 \label{fig:comp_act}
\end{figure*}

\begin{figure}[t!]
  \renewcommand{\figurescale}{1}
  \centering
  \input{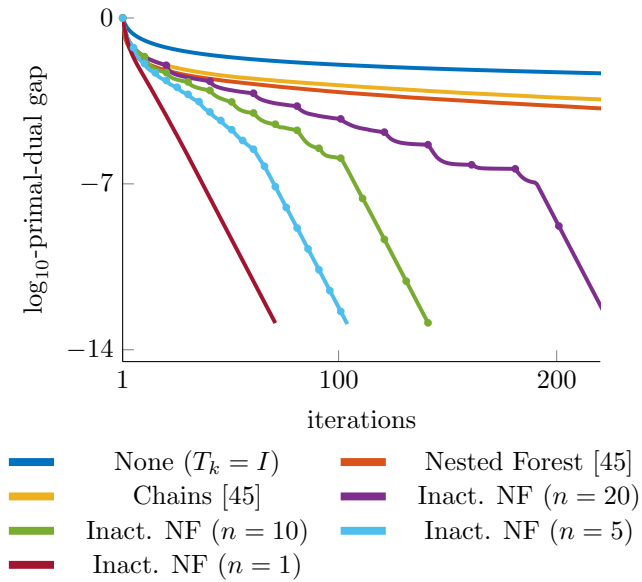}
  \begin{tikzpicture}
    \begin{customlegend}[legend columns=2,legend style={align=center,draw=none,column sep=2ex},
      legend entries={
        None ($T_k = I$),
        Nested Forest~\cite{MoYeWu18},
        Chains~\cite{MoYeWu18},
        Inact. NF ($n=20$),
        Inact. NF ($n=10$),
        Inact. NF ($n=5$),
        Inact. NF ($n=1$),
      }]
      \definecolor{mycolor1}{rgb}{0.00000,0.44700,0.74100}%
      \definecolor{mycolor2}{rgb}{0.85000,0.32500,0.09800}%
      \definecolor{mycolor3}{rgb}{0.92900,0.69400,0.12500}%
      \definecolor{mycolor4}{rgb}{0.49400,0.18400,0.55600}%
      \definecolor{mycolor5}{rgb}{0.46600,0.67400,0.18800}%
      \definecolor{mycolor6}{rgb}{0.3010,0.7450,0.9330}
      \definecolor{mycolor7}{rgb}{0.6350,0.0780,0.1840}
      \addlegendimage{mark=none,solid,line legend,line width=3pt,color=mycolor1}
      \addlegendimage{mark=none,solid,line width=3pt,color=mycolor2}   
      \addlegendimage{mark=none,solid,line width=3pt,color=mycolor3}   
      \addlegendimage{mark=none,solid,line width=3pt,color=mycolor4}
      \addlegendimage{mark=none,solid,line width=3pt,color=mycolor5}
      \addlegendimage{mark=none,solid,line width=3pt,color=mycolor6}
      \addlegendimage{mark=none,solid,line width=3pt,color=mycolor7}
    \end{customlegend}
  \end{tikzpicture}
 \caption{We show $\log_{10}$-primal dual gap vs iterations for PG \eqref{eq:ppg} with various choices of $T_k$. The non-preconditioned choice $T_k = I$ performs the worst, 
 followed by the preconditioners proposed in \cite{MoYeWu18}. We indicate reconditioning by a dot and carry it out every $n$ iterations for $n \in \{ 20, 10, 5, 1\}$. Smaller $n$ leads to an increasingly 
improved performance.}
 \label{fig:theory}
\end{figure}

\section{Applications}
In the following experiments we compare four preconditioning strategies: non-preconditioned $T_k = I$, diagonally scaled $T_k = \diag(K K^\top)$, nested (linear) forest from \cite{MoYeWu18} and the

\subsection{Numerical validation on synthetic data}
As a first numerical example, we consider the fused Lasso~\cite{tibshirani2005sparsity} (also called ROF model in imaging~\cite{Rudin-Osher-Fatemi-92})
\begin{equation}
	\min_{u \in \bR^{\cV}}~ \frac{1}{2} \norm{u-f}^2 + \norm{Ku}_1.
	\label{eq:ROF}
\end{equation}
We solve \eqref{eq:ROF} on random graphs with fixed $|\cV| = 512$ using proximal gradient (PG) with $f$ chosen uniformly random in $[0,1]$. We consider two factors: edge-to-vertex ratio and percentage of active edges at the optimal
solution. For the proposed reconditioning we set the frequency to $n=1$. 

\begin{table*}[t!]
	\centering
	\begin{tabular}{lc||cc|cc|cc|cc|cc}
		\multicolumn{2}{c||}{Instance}& \multicolumn{2}{c|}{None} & \multicolumn{2}{c|}{Diagonal} & \multicolumn{2}{c|}{Nest. Forest} & \multicolumn{2}{c|}{Lin. Forest} & \multicolumn{2}{c}{Inact. NF} \\
		name &  $\frac{|\cA_*|}{|\cE|}$ & it[$10^3$] & time[s] & it[$10^3$] & time[s] & it[$10^3$] & time[s] & it[$10^3$] & time[s] & it[$10^3$] & time[s] \\
		\hline
		rmf-long & 0.02 & -- & -- & 19 & 473 & 12 & 2539 & 18 & \textbf{233.3} & \textbf{1.9} & 474.9 \\
		rmf-wide & 0.19 & -- & -- & 62 & 665 & 27 & 2274 & 43 & 213.1 & \textbf{0.19} & \textbf{18.54} \\
		horse & 0.02 & -- & -- & -- & -- & 2.9 & 340.8 & 37 & 355.6 & \textbf{0.73} & \textbf{155.3}\\
		alue & 0.03 & -- & -- & -- & -- & 4.5 & 117.2 & 100 & 270.9 & \textbf{0.71} & \textbf{155.3}\\
		lux & 0.01 & -- & -- & -- & -- & 13 & 1254 & -- & -- & \textbf{0.40} & \textbf{54.34} \\
		punch & 0.01 & 488 & 968 & -- & -- & 14.9 & 1445 & 203 & 872.1 & \textbf{0.34} & \textbf{62.66} \\
		BVZ* & 0.35 & 27 & 74.0 & 22 & 730 & 1.12 & 434.8 & 0.57 & \textbf{6.59} & \textbf{0.49} & 261\\
		manga* & 0.05 & -- & -- &  -- & -- & 38 & 48591 & 5.3 & \textbf{230} & \textbf{1.41} & 4609 \\
		KZ2 & 0.5 & 419 & 3042 & 1.6 & 159 & 0.43 & 614.1 & 0.6 & \textbf{56.0} & \textbf{0.42} & 965.9 \\
		ferro & 0.09 & 9.25 & 186 & 5.83 & 639.6 & 0.36 & 430.3 & 0.93 & \textbf{86.23} & \textbf{0.27} & 609.3\\
		\hline
	\end{tabular}
\caption{We show the number of iterations and running time to reach a relative primal dual gap less than $10^{-10}$ on \eqref{eq:ROF} on real-world graphs. FISTA with various choices of $T_k$ is used to solve these problems. ``--'' means the algorithm failed to reach the tolerance within $5 \times 10^5$ iterations. ``*'' means that graph has a grid structure. \vspace{0.2cm}}
\label{tab:comp_ROF}
\end{table*}
 The results are shown in Fig.~\ref{fig:comp_act}. For reasonable amounts of active edges at the solution ($30\%$ -- $80\%$) the proposed preconditioning strategy requires orders of magnitude less
iterations to reach a primal-dual gap under $10^{-10}$. Moreover, it is shown that we require the fewest iterations across all scenarios.

In Fig.~\ref{fig:theory} we show $\log_{10}$-primal-dual gap over iterations for PG applied to \eqref{eq:ROF} on a $100 \times 100$ grid graph with different choices of $T_k$ and
moderate regularization strength ($30\%$ of active edges at the optimal solution). The proposed preconditioner outperforms vanilla PG ($T_k = I$) and the recent (fixed) preconditioners proposed in \cite{MoYeWu18}. Reconditioning more often leads to faster convergence, but as recomputing the preconditioner is expensive there is a trade-off 
between reducing the number of iterations and fast updates. In practice, a choice of the reconditioning frequency $n$ between $5$ and $30$ leads to the best performance.

\subsection{Fused Lasso on real-world graphs}
To consider a more realistic scenario, we solve the model \eqref{eq:ROF} on real-world graphs from a popular graphcut benchmark considered in~\cite{goldberg11}. 
Furthermore, instead of using standard PG we used the accelerated FISTA variant~\cite{ChDo2015, attPey2015, Liang2015} with overrelaxation parameter $\beta_k = (k-1)/(k+2)$. Reconditioning
takes place at every 30 iterations. We discard the momentum for one iteration after reconditioning, which improved the stability.
In Table \ref{tab:comp_ROF}, we show the running time and number of iterations of FISTA with non-preconditioned, diagonal preconditioner, nested forest~\cite{MoYeWu18}, linear forest~\cite{MoYeWu18}
and the proposed inactively nested forest. 

Our preconditioner outperforms the other methods in all cases on number of iterations, despite a rather large choice of $n=30$. However, the linear forest
from~\cite{MoYeWu18} perform better with respect to the running time on 5 out of 10 datasets. 
The two datasets with grid structures leads to chain partition on which message-passing is much faster than on trees. 
The sizes of last two graphs are huge ($|\cV| \approx 250.000$, $|\cE| \approx 600.000$) and therefore partitioning is quite expensive. To summarize, the proposed preconditioning
strategy consistently improves the number of iterations, but to ensure a shorter overall running time, an efficient implementation or improved strategy on reconstructing the tree decomposition
might be required.

\begin{figure}
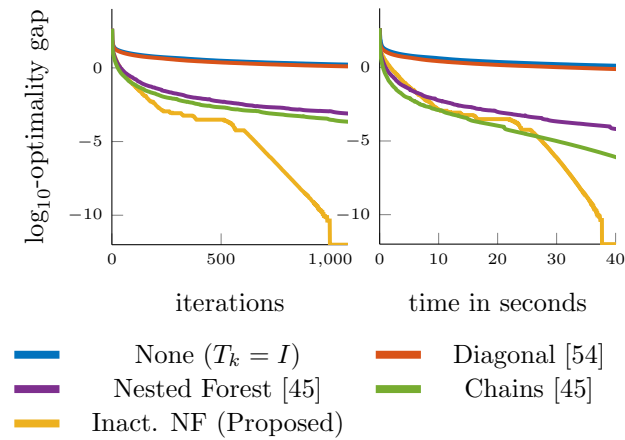

  \centering
  \begin{tabular}{cc}
    \input{pdhg_iterations2_new.tex} &\hspace{-1cm}
    \input{pdhg_time2_new.tex}
  \end{tabular}
  \begin{tikzpicture}
    \begin{customlegend}[legend columns=2,legend style={align=center,draw=none,column sep=2ex},
      legend entries={None ($T_k = I$),
        Diagonal~\cite{PoCh11},
        Nested Forest~\cite{MoYeWu18},
        Chains~\cite{MoYeWu18},
        Inact. NF (Proposed)
      }]
      \definecolor{mycolor1}{rgb}{0.00000,0.44700,0.74100}%
      \definecolor{mycolor2}{rgb}{0.85000,0.32500,0.09800}%
      \definecolor{mycolor3}{rgb}{0.92900,0.69400,0.12500}%
      \definecolor{mycolor4}{rgb}{0.49400,0.18400,0.55600}%
      \definecolor{mycolor5}{rgb}{0.46600,0.67400,0.18800}%
      \addlegendimage{mark=none,solid,line legend,line width=3pt,color=mycolor1}
      \addlegendimage{mark=none,solid,line width=3pt,color=mycolor2}   
      \addlegendimage{mark=none,solid,line width=3pt,color=mycolor4}   
      \addlegendimage{mark=none,solid,line width=3pt,color=mycolor5}   
      \addlegendimage{mark=none,solid,line width=3pt,color=mycolor3}   
    \end{customlegend}
  \end{tikzpicture}

  \caption{$\log_{10}$-optimality gap over iterations (left) and time (right) for PDHG with various preconditioners applied to a TV deconvolution problem.\vspace{-0.0cm}}
  \label{fig:pdhg}
\end{figure}

\subsection{Linear inverse problems}
In this image processing experiment we consider a TV deconvolution problem on a regular 2D grid of size $116 \times 87$.
The data term is given by $G(u) = \frac{1}{2} \norm{Au - f}^2$, where the forward model $A$ is a convolution 
with motion blur kernel with radius $3$. We construct $f$ by applying the forward model and adding Gaussian noise. 
The overall problem is solved using PDHG. 
The primal update is a quadratic problem and we use a few iterations of (warm started) conjugate gradient. 
Considering the size of the problem, we set the reconditioning frequency to $n=5$ for the proposed approach.

In Fig.~\ref{fig:pdhg} we show the $\log_{10}$-optimality gap over iterations and time for various choices of preconditioners. The diagonal 
preconditioner is the one from~\cite{PoCh11} with $\alpha = 1$. The forest preconditioners perform comparably when the accuracy is lower. Once the local convergence regime is entered, the proposed algorithm achieves linear convergence rate. Especially for high accuracies, the proposed 
inactively nested forest reconditioning strategy outperforms the other approaches with
respect to overall running time and iterations.

%% file: active_sparse_new.tex
%
%
\begin{tikzpicture}

\definecolor{mycolor1}{rgb}{0.00000,0.44700,0.74100}%
\definecolor{mycolor2}{rgb}{0.85000,0.32500,0.09800}%
\definecolor{mycolor3}{rgb}{0.92900,0.69400,0.12500}%
\begin{axis}[%
title=$\nicefrac{|\cE|}{|\cV|} \approx 1.3$,
width=1.9in,
height=1.3in,
at={(0.772in,0.516in)},
scale only axis,
xmin=0.01,
xmax=0.99,
ymin=0.5,
ymax=4.5,
axis background/.style={fill=white},
axis x line*=bottom,
axis y line*=left,
xtick={0.01, 0.2, 0.4, 0.6, 0.8, 0.99},
xticklabels={$1\%$,$20\%$,$40\%$,$60\%$,$80\%$,$99\%$},
ylabel={$\log_{10}$-iterations},
ylabel near ticks,
legend style={legend cell align=left, align=left, draw=white!15!black},
scale=\figurescale
]
\addplot [color=mycolor1, line width=1.5pt]
  table[row sep=crcr]{%
0.0139103554868624	3.65896484266443\\
0.0139103554868624	3.78618334556763\\
0.0139103554868624	3.71391035412896\\
0.0139103554868624	3.68313713148301\\
0.0154559505409583	3.47479881880063\\
0.0154559505409583	3.35926616460675\\
0.0216383307573416	3.35850591149024\\
0.0262751159196291	3.3451776165427\\
0.0278207109737249	3.33284226699435\\
0.0324574961360124	3.58103894877217\\
0.0340030911901082	3.36511343162758\\
0.0571870170015456	3.71206014246107\\
0.0664605873261206	3.698535492562\\
0.071097372488408	3.35082927358297\\
0.0741885625965997	3.4970679363985\\
0.0865533230293663	3.49248101012888\\
0.0942812982998454	3.37309598707873\\
0.125193199381762	3.37401474029191\\
0.148377125193199	3.37493155397819\\
0.174652241112828	3.37584643630916\\
0.207109737248841	3.20682587603185\\
0.211746522411128	3.21906033244886\\
0.213292117465224	3.22993768590793\\
0.227202472952087	3.23829706787539\\
0.228748068006182	3.27091163941048\\
0.253477588871716	3.49596039488171\\
0.256568778979907	3.51587384371168\\
0.284389489953632	3.78682237949919\\
0.285935085007728	3.82968964089897\\
0.290571870170015	3.83872319003137\\
0.296754250386399	3.85955857262605\\
0.321483771251932	3.85973856619715\\
0.321483771251932	3.83739902434202\\
0.323029366306028	3.93621204432025\\
0.323029366306028	4.10856502373283\\
0.326120556414219	4.13270784485545\\
0.363214837712519	4.15296067430425\\
0.363214837712519	4.17117043490162\\
0.380216383307573	3.88263836169604\\
0.389489953632148	4.30102999566398\\
0.417310664605873	3.90746510676586\\
0.428129829984544	3.70001106232211\\
0.446676970633694	3.68850880765652\\
0.446676970633694	3.78951020409025\\
0.463678516228748	4.30102999566398\\
0.469860896445131	3.77283492723902\\
0.479134466769706	3.63235604623907\\
0.497681607418856	3.67541169371486\\
0.506955177743431	3.74507479158206\\
0.51468315301391	4.30102999566398\\
0.516228748068006	4.12872228433843\\
0.525502318392581	3.95511023097055\\
0.536321483771252	3.99400903312361\\
0.55177743431221	4.11290649025331\\
0.57032457496136	4.30102999566398\\
0.585780525502318	4.03726709456871\\
0.596599690880989	3.79028516403324\\
0.605873261205564	3.79441833087414\\
0.619783616692427	3.7856856682809\\
0.636785162287481	3.73790792337464\\
0.672333848531685	1.74036268949424\\
0.68160741885626	1.89762709129044\\
0.690880989180835	1.93449845124357\\
0.693972179289026	1.99563519459755\\
0.707882534775889	2.04532297878666\\
0.715610510046368	2.06445798922692\\
0.718701700154559	2.35983548233989\\
0.724884080370943	2.42160392686983\\
0.731066460587326	2.45636603312904\\
0.740340030911901	2.25285303097989\\
0.752704791344668	2.17026171539496\\
0.761978361669243	2.16731733474818\\
0.772797527047913	2.23552844690755\\
0.778979907264297	2.62013605497376\\
0.791344667697063	2.60530504614111\\
0.800618238021638	2.57634135020579\\
0.808346213292117	2.57403126772772\\
0.820710973724884	2.57403126772772\\
0.831530139103555	2.57170883180869\\
0.839258114374034	2.55750720190566\\
0.850077279752705	2.60959440922522\\
0.865533230293663	2.64836001098093\\
0.870170015455951	2.52374646681156\\
0.874806800618238	2.60097289568675\\
0.880989180834621	2.61066016308988\\
0.882534775888717	2.61909333062674\\
0.885625965996909	2.62940959910272\\
0.890262751159196	2.63648789635337\\
0.893353941267388	2.64246452024212\\
0.905718701700155	2.65417654187796\\
0.911901081916538	2.92737036303902\\
0.919629057187017	2.93247376467715\\
0.9258114374034	3.05384642685225\\
0.930448222565688	3.05918461763137\\
0.936630602782071	3.05690485133647\\
0.939721792890263	3.08242630086077\\
0.953632148377125	3.08350261983027\\
0.958268933539413	3.08278537031645\\
0.961360123647604	3.08170727009735\\
0.969088098918083	3.14550717140966\\
};

\addplot [color=mycolor2, line width=1.5pt]
  table[row sep=crcr]{%
0.0139103554868624	3.18383903705642\\
0.0139103554868624	3.77611979905299\\
0.0139103554868624	3.56855371204944\\
0.0139103554868624	3.0948203803548\\
0.0154559505409583	3.08062648692181\\
0.0154559505409583	3.0696680969116\\
0.0216383307573416	3.03302144468291\\
0.0262751159196291	3.02978947083186\\
0.0278207109737249	3.02857125269254\\
0.0324574961360124	3.02366391819779\\
0.0340030911901082	3.00774777800074\\
0.0571870170015456	3.12254352406875\\
0.0664605873261206	3.11727129565576\\
0.071097372488408	3.08919836680515\\
0.0741885625965997	3.06818586174616\\
0.0865533230293663	2.95375969173323\\
0.0942812982998454	2.93801909747621\\
0.125193199381762	2.89486965674525\\
0.148377125193199	2.88874096068289\\
0.174652241112828	2.86746748785905\\
0.207109737248841	2.64246452024212\\
0.211746522411128	2.64738297011462\\
0.213292117465224	2.72181061521255\\
0.227202472952087	2.64246452024212\\
0.228748068006182	2.82151352840477\\
0.253477588871716	2.99738638439731\\
0.256568778979907	2.8733206018154\\
0.284389489953632	3.30920417967041\\
0.285935085007728	3.33061666729444\\
0.290571870170015	3.34419571587144\\
0.296754250386399	3.36716948853468\\
0.321483771251932	3.3732798932775\\
0.321483771251932	3.34596154181314\\
0.323029366306028	3.37912414607039\\
0.323029366306028	3.77283492723902\\
0.326120556414219	3.8065869343278\\
0.363214837712519	3.78774377164647\\
0.363214837712519	3.66247450375031\\
0.380216383307573	3.58613702523079\\
0.389489953632148	3.85381984585676\\
0.417310664605873	3.70303330473369\\
0.428129829984544	3.27160930137883\\
0.446676970633694	3.23628527744803\\
0.446676970633694	3.31111784266251\\
0.463678516228748	3.99708043547173\\
0.469860896445131	3.3732798932775\\
0.479134466769706	3.33223641549144\\
0.497681607418856	3.35276119172383\\
0.506955177743431	3.40636983546927\\
0.51468315301391	4.14116746864623\\
0.516228748068006	3.84404204204102\\
0.525502318392581	3.26481782300954\\
0.536321483771252	3.6756867086994\\
0.55177743431221	3.60959440922522\\
0.57032457496136	3.87005258169354\\
0.585780525502318	3.54629583512144\\
0.596599690880989	3.43632170013973\\
0.605873261205564	3.37657695705651\\
0.619783616692427	3.3384564936046\\
0.636785162287481	3.34888872307144\\
0.672333848531685	1.34242268082221\\
0.68160741885626	1.44715803134222\\
0.690880989180835	1.49136169383427\\
0.693972179289026	1.68124123737559\\
0.707882534775889	1.75587485567249\\
0.715610510046368	1.77815125038364\\
0.718701700154559	1.83884909073726\\
0.724884080370943	1.89762709129044\\
0.731066460587326	1.92941892571429\\
0.740340030911901	1.89762709129044\\
0.752704791344668	1.90308998699194\\
0.761978361669243	1.95904139232109\\
0.772797527047913	1.9731278535997\\
0.778979907264297	2.04921802267018\\
0.791344667697063	2.05307844348342\\
0.800618238021638	2.05690485133647\\
0.808346213292117	2.05690485133647\\
0.820710973724884	2.05690485133647\\
0.831530139103555	2.06069784035361\\
0.839258114374034	2.06069784035361\\
0.850077279752705	2.06069784035361\\
0.865533230293663	2.05690485133647\\
0.870170015455951	2.00860017176192\\
0.874806800618238	2.07918124604762\\
0.880989180834621	2.07918124604762\\
0.882534775888717	2.10720996964787\\
0.885625965996909	2.11058971029925\\
0.890262751159196	2.11394335230684\\
0.893353941267388	2.13353890837022\\
0.905718701700155	2.23299611039215\\
0.911901081916538	2.38021124171161\\
0.919629057187017	2.38738982633873\\
0.9258114374034	2.41830129131975\\
0.930448222565688	2.39967372148104\\
0.936630602782071	2.44715803134222\\
0.939721792890263	2.4814426285023\\
0.953632148377125	2.49692964807321\\
0.958268933539413	2.5092025223311\\
0.961360123647604	2.52244423350632\\
0.969088098918083	2.57403126772772\\
};

\addplot [color=mycolor3, line width=1.5pt]
  table[row sep=crcr]{%
0.0139103554868624	1.7160033436348\\
0.0139103554868624	1.53147891704226\\
0.0139103554868624	1.54406804435028\\
0.0139103554868624	1.69897000433602\\
0.0154559505409583	1.72427586960079\\
0.0154559505409583	1.53147891704226\\
0.0216383307573416	1.70757017609794\\
0.0262751159196291	1.51851393987789\\
0.0278207109737249	1.51851393987789\\
0.0324574961360124	1.7160033436348\\
0.0340030911901082	1.69019608002851\\
0.0571870170015456	1.53147891704226\\
0.0664605873261206	1.67209785793572\\
0.071097372488408	1.51851393987789\\
0.0741885625965997	1.69897000433602\\
0.0865533230293663	1.51851393987789\\
0.0942812982998454	1.69019608002851\\
0.125193199381762	1.72427586960079\\
0.148377125193199	1.7160033436348\\
0.174652241112828	1.67209785793572\\
0.207109737248841	1.50514997831991\\
0.211746522411128	1.50514997831991\\
0.213292117465224	1.49136169383427\\
0.227202472952087	1.51851393987789\\
0.228748068006182	1.54406804435028\\
0.253477588871716	1.51851393987789\\
0.256568778979907	1.5910646070265\\
0.284389489953632	1.50514997831991\\
0.285935085007728	1.55630250076729\\
0.290571870170015	1.57978359661681\\
0.296754250386399	1.53147891704226\\
0.321483771251932	1.51851393987789\\
0.321483771251932	1.51851393987789\\
0.323029366306028	1.49136169383427\\
0.323029366306028	1.50514997831991\\
0.326120556414219	1.49136169383427\\
0.363214837712519	1.49136169383427\\
0.363214837712519	1.50514997831991\\
0.380216383307573	1.50514997831991\\
0.389489953632148	1.50514997831991\\
0.417310664605873	1.51851393987789\\
0.428129829984544	1.50514997831991\\
0.446676970633694	1.54406804435028\\
0.446676970633694	1.54406804435028\\
0.463678516228748	1.54406804435028\\
0.469860896445131	1.54406804435028\\
0.479134466769706	1.54406804435028\\
0.497681607418856	1.55630250076729\\
0.506955177743431	1.56820172406699\\
0.51468315301391	1.6232492903979\\
0.516228748068006	1.61278385671974\\
0.525502318392581	1.60205999132796\\
0.536321483771252	1.61278385671974\\
0.55177743431221	1.60205999132796\\
0.57032457496136	1.72427586960079\\
0.585780525502318	1.69897000433602\\
0.596599690880989	1.69897000433602\\
0.605873261205564	1.7160033436348\\
0.619783616692427	1.70757017609794\\
0.636785162287481	1.55630250076729\\
0.672333848531685	1.14612803567824\\
0.68160741885626	1.23044892137827\\
0.690880989180835	1.25527250510331\\
0.693972179289026	1.27875360095283\\
0.707882534775889	1.32221929473392\\
0.715610510046368	1.32221929473392\\
0.718701700154559	1.32221929473392\\
0.724884080370943	1.34242268082221\\
0.731066460587326	1.34242268082221\\
0.740340030911901	1.34242268082221\\
0.752704791344668	1.36172783601759\\
0.761978361669243	1.39794000867204\\
0.772797527047913	1.38021124171161\\
0.778979907264297	1.38021124171161\\
0.791344667697063	1.41497334797082\\
0.800618238021638	1.38021124171161\\
0.808346213292117	1.39794000867204\\
0.820710973724884	1.39794000867204\\
0.831530139103555	1.39794000867204\\
0.839258114374034	1.41497334797082\\
0.850077279752705	1.41497334797082\\
0.865533230293663	1.41497334797082\\
0.870170015455951	1.41497334797082\\
0.874806800618238	1.43136376415899\\
0.880989180834621	1.44715803134222\\
0.882534775888717	1.43136376415899\\
0.885625965996909	1.43136376415899\\
0.890262751159196	1.44715803134222\\
0.893353941267388	1.44715803134222\\
0.905718701700155	1.44715803134222\\
0.911901081916538	1.46239799789896\\
0.919629057187017	1.43136376415899\\
0.9258114374034	1.44715803134222\\
0.930448222565688	1.47712125471966\\
0.936630602782071	1.46239799789896\\
0.939721792890263	1.46239799789896\\
0.953632148377125	1.50514997831991\\
0.958268933539413	1.44715803134222\\
0.961360123647604	1.47712125471966\\
0.969088098918083	1.46239799789896\\
};

\end{axis}
\end{tikzpicture}%

%% file: active_medium_new.tex
%
%
\begin{tikzpicture}

\definecolor{mycolor1}{rgb}{0.00000,0.44700,0.74100}%
\definecolor{mycolor2}{rgb}{0.85000,0.32500,0.09800}%
\definecolor{mycolor3}{rgb}{0.92900,0.69400,0.12500}%
\begin{axis}[%
title=$\nicefrac{|\cE|}{|\cV|} \approx 2.8$,
width=1.9in,
height=1.3in,
at={(0.772in,0.516in)},
scale only axis,
xmin=0.01,
xmax=0.99,
ymin=1.5,
ymax=4.5,
axis background/.style={fill=white},
axis x line*=bottom,
axis y line*=left,
xtick={0.01, 0.2, 0.4, 0.6, 0.8, 0.99},
xticklabels={$1\%$,$20\%$,$40\%$,$60\%$,$80\%$,$99\%$},
legend style={legend cell align=left, align=left, draw=white!15!black},
scale=\figurescale
]

\addplot [color=mycolor1, line width=1.5pt]
  table[row sep=crcr]{%
0.0091164095371669	3.48287358360875\\
0.00981767180925666	3.23350376034113\\
0.0105189340813464	3.19838213000829\\
0.0105189340813464	3.37438169805088\\
0.0112201963534362	2.93094903116752\\
0.0112201963534362	2.86093662070009\\
0.0112201963534362	2.84509804001426\\
0.0126227208976157	2.78675142214556\\
0.0133239831697055	2.72181061521255\\
0.0133239831697055	2.71933128698373\\
0.014726507713885	2.71516735784846\\
0.0154277699859748	2.69019608002851\\
0.0161290322580645	2.65801139665711\\
0.0161290322580645	2.6170003411209\\
0.0168302945301543	2.60959440922522\\
0.0182328190743338	2.88252453795488\\
0.0196353436185133	2.88309335857569\\
0.0196353436185133	2.88366143515362\\
0.0196353436185133	2.88366143515362\\
0.0224403927068724	2.88366143515362\\
0.0231416549789621	2.88252453795488\\
0.0238429172510519	2.98136550907854\\
0.0245441795231417	2.97589113640179\\
0.0273492286115007	2.96567197122011\\
0.0280504908835905	2.89097959698969\\
0.032258064516129	2.88986172125819\\
0.0329593267882188	2.88761730033574\\
0.0343618513323983	2.88536122003151\\
0.0399719495091164	2.8819549713396\\
0.0406732117812062	2.87621784059164\\
0.0483870967741935	2.8668778143375\\
0.0511921458625526	2.83505610172012\\
0.0561009817671809	2.71850168886727\\
0.0568022440392707	2.7160033436348\\
0.061711079943899	2.71180722904119\\
0.0624123422159888	2.7041505168398\\
0.0771388499298738	2.68214507637383\\
0.0785413744740533	2.66745295288995\\
0.091164095371669	2.66838591669\\
0.0960729312762973	2.66838591669\\
0.103786816269285	2.66931688056611\\
0.11711079943899	2.66931688056611\\
0.119214586255259	2.67024585307412\\
0.133941093969144	2.67024585307412\\
0.150070126227209	2.67117284271508\\
0.157784011220196	2.67117284271508\\
0.178821879382889	2.67209785793572\\
0.199158485273492	2.67209785793572\\
0.246844319775596	2.6730209071289\\
0.268583450210379	2.6730209071289\\
0.306451612903226	2.67394199863409\\
0.354838709677419	2.67394199863409\\
0.403927068723703	2.67394199863409\\
0.455820476858345	2.67486114073781\\
0.462833099579243	4.25935492730803\\
0.494389901823282	4.11537749427912\\
0.5	4.05736176298504\\
0.502805049088359	3.82046419057768\\
0.509817671809257	3.77655591070326\\
0.516129032258065	3.71424591101789\\
0.537868162692847	3.67897337591977\\
0.542776998597475	3.66698571832966\\
0.562412342215989	3.86254876952479\\
0.575736325385694	3.88637790675857\\
0.577840112201964	4.30102999566398\\
0.604488078541374	4.03059972196595\\
0.613604488078541	3.9897167199481\\
0.638849929873773	4.1424833236595\\
0.640953716690042	4.24750682414992\\
0.667601683029453	3.87788942537148\\
0.680925666199159	4.01857568346725\\
0.691444600280505	3.71011736511182\\
0.694249649368864	3.51693180886801\\
0.700561009817672	3.56925683332861\\
0.7054698457223	3.36492603378998\\
0.711781206171108	3.2771506139638\\
0.715988779803647	3.17840134153376\\
0.718092566619916	3.14238946611884\\
0.723702664796634	3.09061070782841\\
0.724403927068724	3.07809415040641\\
0.726507713884993	3.29754166781816\\
0.731416549789621	3.28194193344082\\
0.732819074333801	3.42797271360821\\
0.757363253856942	2.19865708695442\\
0.773492286115007	2.17897694729317\\
0.779803646563815	2.55022835305509\\
0.788218793828892	2.62531245096167\\
0.806451612903226	2.81954393554187\\
0.82258064516129	2.83058866868514\\
0.831697054698457	2.84260923961056\\
0.848527349228612	2.79934054945358\\
0.865357643758766	2.91487181754005\\
0.875175315568022	2.90363251608424\\
0.884992987377279	3.09725730969342\\
0.897615708274895	3.29797924415936\\
0.922159887798036	3.64982146322457\\
0.934782608695652	3.44451320633404\\
0.950210378681627	3.53147891704225\\
0.957924263674614	3.48129927333286\\
0.97054698457223	3.72599325892472\\
};

\addplot [color=mycolor2, line width=1.5pt]
  table[row sep=crcr]{%
0.0091164095371669	3.05269394192497\\
0.00981767180925666	3.00432137378264\\
0.0105189340813464	3.13956426617585\\
0.0105189340813464	3.06744284277638\\
0.0112201963534362	2.73399928653839\\
0.0112201963534362	2.58433122436753\\
0.0112201963534362	2.54032947479087\\
0.0126227208976157	2.52762990087134\\
0.0133239831697055	2.50514997831991\\
0.0133239831697055	2.48429983934679\\
0.014726507713885	2.4814426285023\\
0.0154277699859748	2.46686762035411\\
0.0161290322580645	2.36735592102602\\
0.0161290322580645	2.30102999566398\\
0.0168302945301543	2.29885307640971\\
0.0182328190743338	2.39093510710338\\
0.0196353436185133	2.38916608436453\\
0.0196353436185133	2.47421626407626\\
0.0196353436185133	2.47421626407626\\
0.0224403927068724	2.47129171105894\\
0.0231416549789621	2.46834733041216\\
0.0238429172510519	2.46089784275655\\
0.0245441795231417	2.45788189673399\\
0.0273492286115007	2.46686762035411\\
0.0280504908835905	2.46686762035411\\
0.032258064516129	2.46834733041216\\
0.0329593267882188	2.46834733041216\\
0.0343618513323983	2.46982201597816\\
0.0399719495091164	2.46982201597816\\
0.0406732117812062	2.47129171105894\\
0.0483870967741935	2.47129171105894\\
0.0511921458625526	2.47129171105894\\
0.0561009817671809	2.47129171105894\\
0.0568022440392707	2.47129171105894\\
0.061711079943899	2.47129171105894\\
0.0624123422159888	2.47129171105894\\
0.0771388499298738	2.47129171105894\\
0.0785413744740533	2.47129171105894\\
0.091164095371669	2.46538285144842\\
0.0960729312762973	2.46538285144842\\
0.103786816269285	2.46538285144842\\
0.11711079943899	2.46538285144842\\
0.119214586255259	2.46538285144842\\
0.133941093969144	2.46389298898591\\
0.150070126227209	2.46389298898591\\
0.157784011220196	2.46389298898591\\
0.178821879382889	2.46239799789896\\
0.199158485273492	2.46239799789896\\
0.246844319775596	2.46089784275655\\
0.268583450210379	2.45939248775923\\
0.306451612903226	2.45788189673399\\
0.354838709677419	2.45636603312904\\
0.403927068723703	2.45484486000851\\
0.455820476858345	2.45178643552429\\
0.462833099579243	3.91179659043725\\
0.494389901823282	3.84023159495811\\
0.5	3.79070728732768\\
0.502805049088359	3.8077379220141\\
0.509817671809257	3.76930346018908\\
0.516129032258065	3.75572244490346\\
0.537868162692847	3.73687436164842\\
0.542776998597475	3.73407940728059\\
0.562412342215989	3.7435881501599\\
0.575736325385694	3.7419390777292\\
0.577840112201964	4.17032039987338\\
0.604488078541374	3.79712898779655\\
0.613604488078541	3.61352470285365\\
0.638849929873773	3.80009818017478\\
0.640953716690042	4.11451090484887\\
0.667601683029453	3.63255851453267\\
0.680925666199159	4.05637117947553\\
0.691444600280505	3.75004531201177\\
0.694249649368864	3.41396997174806\\
0.700561009817672	3.32613095671079\\
0.7054698457223	3.19728055812562\\
0.711781206171108	3.12417805547468\\
0.715988779803647	3.03542973818455\\
0.718092566619916	2.98900461569854\\
0.723702664796634	2.96378782734556\\
0.724403927068724	2.93601079571521\\
0.726507713884993	2.97451169273733\\
0.731416549789621	2.96094619573383\\
0.732819074333801	2.95712819767681\\
0.757363253856942	2.15228834438306\\
0.773492286115007	2\\
0.779803646563815	2.40654018043396\\
0.788218793828892	2.43933269383026\\
0.806451612903226	2.50514997831991\\
0.82258064516129	2.44715803134222\\
0.831697054698457	2.48429983934679\\
0.848527349228612	2.66275783168157\\
0.865357643758766	2.64738297011462\\
0.875175315568022	2.63948648926859\\
0.884992987377279	2.82736927305383\\
0.897615708274895	3.05576046468773\\
0.922159887798036	3.291590825658\\
0.934782608695652	3.28600712207947\\
0.950210378681627	3.2043913319193\\
0.957924263674614	3.13956426617585\\
0.97054698457223	3.32980452216407\\
};

\addplot [color=mycolor3, line width=1.5pt]
  table[row sep=crcr]{%
0.0091164095371669	1.81954393554187\\
0.00981767180925666	1.81291335664286\\
0.0105189340813464	1.80617997398389\\
0.0105189340813464	1.80617997398389\\
0.0112201963534362	1.80617997398389\\
0.0112201963534362	1.79934054945358\\
0.0112201963534362	1.79239168949825\\
0.0126227208976157	1.78532983501077\\
0.0133239831697055	1.77085201164214\\
0.0133239831697055	1.78532983501077\\
0.014726507713885	1.77815125038364\\
0.0154277699859748	1.78532983501077\\
0.0161290322580645	1.77815125038364\\
0.0161290322580645	1.77815125038364\\
0.0168302945301543	1.77085201164214\\
0.0182328190743338	1.77815125038364\\
0.0196353436185133	1.77085201164214\\
0.0196353436185133	1.77085201164214\\
0.0196353436185133	1.76342799356294\\
0.0224403927068724	1.77815125038364\\
0.0231416549789621	1.75587485567249\\
0.0238429172510519	1.76342799356294\\
0.0245441795231417	1.77085201164214\\
0.0273492286115007	1.77815125038364\\
0.0280504908835905	1.78532983501077\\
0.032258064516129	1.78532983501077\\
0.0329593267882188	1.79239168949825\\
0.0343618513323983	1.78532983501077\\
0.0399719495091164	1.78532983501077\\
0.0406732117812062	1.78532983501077\\
0.0483870967741935	1.78532983501077\\
0.0511921458625526	1.79239168949825\\
0.0561009817671809	1.79239168949825\\
0.0568022440392707	1.79239168949825\\
0.061711079943899	1.79934054945358\\
0.0624123422159888	1.78532983501077\\
0.0771388499298738	1.79239168949825\\
0.0785413744740533	1.79239168949825\\
0.091164095371669	1.89209460269048\\
0.0960729312762973	1.79239168949825\\
0.103786816269285	1.79239168949825\\
0.11711079943899	1.79239168949825\\
0.119214586255259	1.79239168949825\\
0.133941093969144	1.79239168949825\\
0.150070126227209	1.79239168949825\\
0.157784011220196	1.79239168949825\\
0.178821879382889	1.79934054945358\\
0.199158485273492	1.79239168949825\\
0.246844319775596	1.79934054945358\\
0.268583450210379	1.79239168949825\\
0.306451612903226	1.79239168949825\\
0.354838709677419	1.80617997398389\\
0.403927068723703	1.80617997398389\\
0.455820476858345	1.80617997398389\\
0.462833099579243	1.82607480270083\\
0.494389901823282	1.83250891270624\\
0.5	1.83250891270624\\
0.502805049088359	1.83250891270624\\
0.509817671809257	1.83250891270624\\
0.516129032258065	1.84509804001426\\
0.537868162692847	1.84509804001426\\
0.542776998597475	1.84509804001426\\
0.562412342215989	1.85733249643127\\
0.575736325385694	1.85733249643127\\
0.577840112201964	1.86332286012046\\
0.604488078541374	1.86332286012046\\
0.613604488078541	1.88649072517248\\
0.638849929873773	1.93951925261862\\
0.640953716690042	1.88081359228079\\
0.667601683029453	1.90308998699194\\
0.680925666199159	1.89209460269048\\
0.691444600280505	1.86923171973098\\
0.694249649368864	1.86332286012046\\
0.700561009817672	1.85733249643127\\
0.7054698457223	1.84509804001426\\
0.711781206171108	1.83250891270624\\
0.715988779803647	1.83250891270624\\
0.718092566619916	1.83884909073726\\
0.723702664796634	1.83250891270624\\
0.724403927068724	1.83250891270624\\
0.726507713884993	1.82607480270083\\
0.731416549789621	1.83250891270624\\
0.732819074333801	1.82607480270083\\
0.757363253856942	1.60205999132796\\
0.773492286115007	1.66275783168157\\
0.779803646563815	1.70757017609794\\
0.788218793828892	1.72427586960079\\
0.806451612903226	1.7481880270062\\
0.82258064516129	1.75587485567249\\
0.831697054698457	1.77085201164214\\
0.848527349228612	1.77815125038364\\
0.865357643758766	1.77815125038364\\
0.875175315568022	1.79239168949825\\
0.884992987377279	1.79239168949825\\
0.897615708274895	1.79934054945358\\
0.922159887798036	1.80617997398389\\
0.934782608695652	1.81291335664286\\
0.950210378681627	1.81291335664286\\
0.957924263674614	1.81954393554187\\
0.97054698457223	1.81954393554187\\
};

\end{axis}
\end{tikzpicture}%

%% file: active_dense_new.tex
%
%
\begin{tikzpicture}

\definecolor{mycolor1}{rgb}{0.00000,0.44700,0.74100}%
\definecolor{mycolor2}{rgb}{0.85000,0.32500,0.09800}%
\definecolor{mycolor3}{rgb}{0.92900,0.69400,0.12500}%
\begin{axis}[%
title=$\nicefrac{|\cE|}{|\cV|} \approx 5.0$,
width=1.9in,
height=1.3in,
at={(0.772in,0.516in)},
scale only axis,
xmin=0.01,
xmax=0.99,
ymin=1,
ymax=4.5,
axis background/.style={fill=white},
axis x line*=bottom,
axis y line*=left,
ytick={1.5,2.5,3.5,4.5},
xtick={0.01, 0.2, 0.4, 0.6, 0.8, 0.99},
xticklabels={$1\%$,$20\%$,$40\%$,$60\%$,$80\%$,$99\%$},
legend style={legend cell align=left, align=left, draw=white!15!black},
scale=\figurescale
]
\addplot [color=mycolor1, line width=1.5pt]
  table[row sep=crcr]{%
0	2.15836249209525\\
0.000386100386100386	2.21218760440396\\
0.00154440154440154	2.27875360095283\\
0.00501930501930502	2.58433122436753\\
0.0108108108108108	2.65127801399814\\
0.0864864864864865	2.49276038902684\\
0.12007722007722	2.67851837904011\\
0.130501930501931	2.74350976472843\\
0.13976833976834	2.73639650227664\\
0.145559845559846	2.71850168886727\\
0.153281853281853	2.73078227566639\\
0.172200772200772	2.74429298312268\\
0.176833976833977	2.7427251313047\\
0.183783783783784	2.78247262416629\\
0.194980694980695	3.07773117965239\\
0.217760617760618	3.21563756343506\\
0.231660231660232	3.19117145572856\\
0.248648648648649	3.51308436046514\\
0.253667953667954	3.61468634228201\\
0.297297297297297	3.46448954743397\\
0.324324324324324	3.7592900330243\\
0.337065637065637	3.57818060962778\\
0.361776061776062	3.78738962135211\\
0.367181467181467	4.00697919057428\\
0.435521235521236	3.75709222011893\\
0.467953667953668	3.75769962508774\\
0.508880308880309	3.97872817713849\\
0.553667953667954	3.67906431812131\\
0.595366795366795	3.55303301620244\\
0.655598455598456	4.30102999566398\\
0.715444015444015	3.93449845124357\\
0.748648648648649	3.33183204443625\\
0.786486486486486	3.23879856271392\\
0.829343629343629	3.78433194802215\\
0.839382239382239	3.28780172993023\\
0.851351351351351	3.18836592606315\\
0.861776061776062	3.19451434188247\\
0.868725868725869	2.86864443839483\\
0.877606177606178	2.76267856372744\\
0.884169884169884	2.74585519517373\\
0.893436293436293	2.82930377283102\\
0.896911196911197	2.82607480270083\\
0.901158301158301	2.8055008581584\\
0.908108108108108	3.02857125269254\\
0.910810810810811	3.05880548667591\\
0.921621621621622	2.80753502806885\\
0.929343629343629	2.77158748088126\\
0.938996138996139	2.6170003411209\\
0.945945945945946	2.60422605308447\\
0.949034749034749	2.59879050676312\\
0.95019305019305	2.5910646070265\\
0.954440154440154	2.59659709562646\\
0.956370656370656	2.58994960132571\\
0.956370656370656	2.72672720902657\\
0.961776061776062	3.00603795499732\\
0.962548262548263	2.43616264704076\\
0.967953667953668	2.44715803134222\\
0.972586872586873	2.45484486000851\\
0.974903474903475	2.46089784275655\\
0.983783783783784	2.46538285144842\\
};

\addplot [color=mycolor2, line width=1.5pt]
  table[row sep=crcr]{%
0	2.05307844348342\\
0.000386100386100386	2.07554696139253\\
0.00154440154440154	2.13672056715641\\
0.00501930501930502	2.51587384371168\\
0.0108108108108108	2.45636603312904\\
0.0864864864864865	2.20682587603185\\
0.12007722007722	2.38201704257487\\
0.130501930501931	2.42651126136458\\
0.13976833976834	2.41830129131975\\
0.145559845559846	2.45939248775923\\
0.153281853281853	2.53402610605613\\
0.172200772200772	2.60638136511061\\
0.176833976833977	2.59549622182557\\
0.183783783783784	2.59439255037543\\
0.194980694980695	2.99738638439731\\
0.217760617760618	3.08098704691089\\
0.231660231660232	3.09898963940118\\
0.248648648648649	3.13956426617585\\
0.253667953667954	3.4334497937616\\
0.297297297297297	3.31618009889345\\
0.324324324324324	3.49554433754645\\
0.337065637065637	3.30920417967041\\
0.361776061776062	3.48770386316373\\
0.367181467181467	3.71323846154566\\
0.435521235521236	3.54814363743485\\
0.467953667953668	3.62376600013393\\
0.508880308880309	3.87644878687834\\
0.553667953667954	3.75981887737483\\
0.595366795366795	3.50023647482564\\
0.655598455598456	4.03686832998106\\
0.715444015444015	3.79643555881017\\
0.748648648648649	3.16731733474818\\
0.786486486486486	3.11159852488039\\
0.829343629343629	3.56193576331378\\
0.839382239382239	3.11892575282578\\
0.851351351351351	3.02571538390134\\
0.861776061776062	3.01410032151962\\
0.868725868725869	3.17231096852195\\
0.877606177606178	3.04960561259497\\
0.884169884169884	2.86510397464113\\
0.893436293436293	2.86272752831797\\
0.896911196911197	2.85733249643127\\
0.901158301158301	2.85064623518307\\
0.908108108108108	2.9216864754836\\
0.910810810810811	2.83122969386706\\
0.921621621621622	2.78103693862113\\
0.929343629343629	2.6954816764902\\
0.938996138996139	2.69196510276736\\
0.945945945945946	2.68574173860226\\
0.949034749034749	2.66838591669\\
0.95019305019305	2.58433122436753\\
0.954440154440154	2.57749179983723\\
0.956370656370656	2.63447727016073\\
0.956370656370656	2.63346845557959\\
0.961776061776062	2.77815125038364\\
0.962548262548263	2.2380461031288\\
0.967953667953668	2.2380461031288\\
0.972586872586873	2.23299611039215\\
0.974903474903475	2.20951501454263\\
0.983783783783784	2.2148438480477\\
};

\addplot [color=mycolor3, line width=1.5pt]
  table[row sep=crcr]{%
0	1.77085201164214\\
0.000386100386100386	1.80617997398389\\
0.00154440154440154	1.83250891270624\\
0.00501930501930502	1.85125834871908\\
0.0108108108108108	1.86923171973098\\
0.0864864864864865	1.88081359228079\\
0.12007722007722	1.89209460269048\\
0.130501930501931	1.89762709129044\\
0.13976833976834	1.92427928606188\\
0.145559845559846	1.92941892571429\\
0.153281853281853	1.9731278535997\\
0.172200772200772	2\\
0.176833976833977	2.00860017176192\\
0.183783783783784	1.99563519459755\\
0.194980694980695	1.96378782734556\\
0.217760617760618	1.96848294855394\\
0.231660231660232	1.98227123303957\\
0.248648648648649	1.99563519459755\\
0.253667953667954	2\\
0.297297297297297	2.04921802267018\\
0.324324324324324	2.05307844348342\\
0.337065637065637	2.05690485133647\\
0.361776061776062	2.05690485133647\\
0.367181467181467	2\\
0.435521235521236	2.04921802267018\\
0.467953667953668	2.04921802267018\\
0.508880308880309	2.04532297878666\\
0.553667953667954	2.07188200730613\\
0.595366795366795	2.02938377768521\\
0.655598455598456	2.02530586526477\\
0.715444015444015	2\\
0.748648648648649	2.02530586526477\\
0.786486486486486	1.9731278535997\\
0.829343629343629	1.97772360528885\\
0.839382239382239	1.96378782734556\\
0.851351351351351	1.96378782734556\\
0.861776061776062	1.95904139232109\\
0.868725868725869	1.95424250943932\\
0.877606177606178	1.95424250943932\\
0.884169884169884	1.94448267215017\\
0.893436293436293	1.92427928606188\\
0.896911196911197	1.90848501887865\\
0.901158301158301	1.90848501887865\\
0.908108108108108	1.89209460269048\\
0.910810810810811	1.91381385238372\\
0.921621621621622	1.91907809237607\\
0.929343629343629	1.92427928606188\\
0.938996138996139	1.91381385238372\\
0.945945945945946	1.89209460269048\\
0.949034749034749	1.93951925261862\\
0.95019305019305	1.91907809237607\\
0.954440154440154	1.89762709129044\\
0.956370656370656	1.88081359228079\\
0.956370656370656	1.9731278535997\\
0.961776061776062	1.72427586960079\\
0.962548262548263	1.7160033436348\\
0.967953667953668	1.73239375982297\\
0.972586872586873	1.74036268949424\\
0.974903474903475	1.7481880270062\\
0.983783783783784	1.75587485567249\\
};

\end{axis}
\end{tikzpicture}%

%% file: conclusion.tex
\section{Discussion and conclusion}
We presented an efficient reconditioning strategy for proximal
algorithms on graphs. By relying on a sharp analysis of the local linear convergence
rate we proposed an edge partitioning of the graph into forests which provably boosts
the linear convergence rate. The scaled dual updates are still efficiently computable thanks to a
message-passing algorithm on trees.

While one is tempted to commit to a super-linearly convergent solver once the
optimal active set is identified (as e.g., mentioned in~\cite{Liang14,Liang2015,liang2018local,nutini17a}),
it is unfortunately difficult to verify in practice whether the current active set is the optimal. Furthermore, as
observed in the numerical experiments, the adaptive preconditioning strategy practically improves the convergence
also \emph{before} the local linear convergence regime is entered. The result suggests that
local convergence analysis can serve as a practical guideline for constructing preconditioners for
proximal algorithms.

%% file: supplementary.tex
\twocolumn[

\aistatstitle{Optimization of Graph Total Variation
	via Active-Set-based Combinatorial Reconditioning \\ 
	--- Supplementary Material ---}

\aistatsauthor{ Zhenzhang Ye \And Thomas M\"ollenhoff \And  Tao Wu \And Daniel Cremers }

\aistatsaddress{ TU Munich \\ \href{mailto:zhenzhang.ye@tum.de}{zhenzhang.ye@tum.de} \And TU Munich \\ \href{mailto:thomas.moellenhoff@tum.de}{thomas.moellenhoff@tum.de} \And TU Munich \\ \href{mailto:tao.wu@tum.de}{tao.wu@tum.de} \And TU Munich \\ \href{mailto:cremers@tum.de}{cremers@tum.de}} 

\begin{lemma}
  Let $h$ be $C^2$ with $l_h I \preceq \nabla^2 h(\cdot) \preceq L_h I$ for some constants $l_h,\,L_h>0$. Then
 the gradient descent on $\min_x~h(Ax+b)$ with step size $1/t = 2 / (L_h \sigma_{\text{max}}(A)^2 + l_h \sigma_{\text{min}>0}(A)^2)$ satisfies
  \begin{equation}
    \norm{x^{k+1} - x^*} \leq \frac{\varphi - 1}{\varphi + 1} \norm{x^k - x^*},
  \end{equation}
  with $\varphi =  \kappa(A)^2 \cdot \kappa(h)$, $\kappa(h) := L_h / l_h$.
  \label{lm:cg}
\end{lemma}

\input{appendix.tex}

]

%% file: appendix.tex

\label{sec:appendix}

\begin{proof}[Proof of Lemma~\ref{lm:cg}]
Clearly $t > (L_h \cdot \lambda_{\text{max}}(A^\top A)) / 2$, so classical theory (e.g.~\cite{BaCo11}) guarantees $x^k \to x^*$.
First note that
\iali{
  x^{k+1} - x^k = -\frac1t A^\top \nabla h(Ax^k + b) \in \ran A^\top,
}
and consequently $x^{k} - x^0 \in (\ker A)^\perp$, also $x^* - x^k \in (\ker A)^\perp$.
Inserting the gradient step for $x^{k+1}$ yields
\begin{equation}
  \begin{aligned}
    &\norm{x^{k+1} - x^*}  
    = \norm{x^k - x^* - \frac1t A^\top (\nabla h(Ax^k + b) - \nabla h(Ax^* + b))}.
  \end{aligned}
\end{equation}

From the mean value theorem it follows
\begin{equation}
  \begin{aligned}
    &\nabla h(Ax^k + b) - \nabla h(Ax^* + b) = M (Ax^k - Ax^*),
  \end{aligned}
\end{equation}
with $M = \int_0^1 \nabla^2 h(Ax^* + b + \alpha (Ax^* - Ax^k)) \mathrm{d}\alpha$.
Since $l_h I \preceq \nabla^2 h(\cdot) \preceq L_h I$ we have $l_h I \preceq M \preceq L_h I$.

This yields due to $x^k - x^* \in (\ker A)^\perp$ that
\begin{equation}
  \begin{aligned}
    &\norm{x^{k+1} - x^*} = \norm{(I - \frac1t A^\top M A)(x^k - x^*)}
    \leq \max \{ | 1 - l_h \sigma_{\min>0}(A)^2 / t |, | 1 - L_h \sigma_{\max}(A)^2 / t | \} \cdot \norm{x^k - x^*}.
  \end{aligned}
\end{equation}
The choice $t = (l_h \sigma_{\min>0}(A)^2 + L_h \sigma_{\max}(A)^2) / 2$ minimizes the above rate and yields the desired result.
\end{proof}